\documentclass[11pt]{amsart}
\usepackage[utf8]{inputenc}
\usepackage[left=2cm, right=2cm, top=2cm, bottom=2cm]{geometry}
\usepackage{mdframed, mathtools, esint, amsmath, enumitem, amscd, amssymb, indentfirst, latexsym, multicol, verbatim, mdframed, amsthm}
\usepackage{mathtools}

\mathtoolsset{showonlyrefs}

\usepackage[ddmmyyyy]{datetime}
\usepackage[dvipsnames,table,xcdraw]{xcolor}
\usepackage[T1]{fontenc}
\usepackage{fixltx2e}
\usepackage[sort, numbers]{natbib}
\usepackage[normalem]{ulem}
\makeatletter\def\SL@eqnlefttext #1{\hbox to 0pt{\kern 60pt %or something else
\llap{\SL@margintext{#1}\quad}\hss}}
\makeatother

\newcommand{\overbar}[1]{\mkern 1.5mu\overline{\mkern-1.5mu#1\mkern-1.5mu}\mkern 1.5mu}

\newcommand{\R}{\mathbb{R}}

\newcommand*\dd{\mathop{}\!\mathrm{d}}

\renewcommand{\div}{\text{div}} %Dywergencja, słownie
\renewcommand{\phi}{\varphi}

\renewcommand{\mapsto}{\longmapsto}
\newcommand{\xc}{\|x\|_C}
\newcommand{\vc}{\|v\|_C}

\renewcommand{\epsilon}{\varepsilon}

\newcommand{\mf}[1]{{\color{Cerulean} \textbf{M.F.:} #1}}
\newlist{steps}{enumerate}{1}
\setlist[steps, 1]{label = Step \arabic*:}

\newtheorem{theo}{Theorem}[section]
\newtheorem{prop}[theo]{Proposition}
\newtheorem{lem}[theo]{Lemma}
\newtheorem{rem}[theo]{Remark}

\newtheorem{definition}[theo]{Definition}

\numberwithin{equation}{section}

\title[Alignment with nonlinear velocity couplings]{Alignment with nonlinear velocity couplings: collision-avoidance and  micro-to-macro mean-field limits}
\author{Young-Pil Choi, Michał Fabisiak, Jan Peszek}
\date{\today}

\begin{document}

 \thanks{\textbf{Acknowledgement:} The paper has been partly supported by the Polish National Science Centre’s Grant No. 2018/31/D/ST1/02313 (SONATA). YPC's work has been supported by NRF grants No. 2022R1A2C1002820 and No. RS-2024-00406821. JP's work has been additionally supported by the Polish National Science Centre's Grant No. 2018/30/M/ST1/00340 (HARMONIA) and partly by the University of Warsaw
program IDUB Nowe Idee 3A}

% \subjclass{35Q70,
% 35D30, 
% 35L81, 
% 35Q83 }


\maketitle

% \begin{abstract}
%     We prove that certain types of measure-valued mappings are monokinetic i.e. the distribution of velocity is concentrated in a Dirac mass. These include weak measure-valued solutions to the strongly singular Cucker-Smale model with singularity of order $\alpha$ greater or equal to the dimension of the ambient space. Consequently, we are able to answer a couple of open questions related to the singular Cucker-Smale model. First,  we prove that weak measure-valued solutions to the strongly singular Cucker-Smale kinetic equation are monokinetic, under very mild assumptions that they are uniformly compactly supported and weakly continuous in time. This can be interpreted as a rigorous derivation of the macroscopic fractional Euler-alignment system from kinetic Cucker-Smale equation without the need to perform any hydrodynamical limit.
%     This suggests superior suitability of the macroscopic framework to describe large-crowd limits of strongly singular Cucker-Smale dynamics.
%     Second, we perform a direct micro- to macroscopic mean-field limit from the Cucker-Smale particle system to the fractional Euler-alignment model. This leads to the final result -- existence of weak solutions to the fractional Euler-alignment system with almost arbitrary initial data in $\R^1$, including the possibility of vacuum. Existence can be extended to $\R^2$ under the a priori assumption that the density of the mean-field limit has no atoms.
% \end{abstract}

%\tableofcontents

\begin{abstract} 
We investigate the pressureless fractional Euler-alignment system with nonlinear velocity couplings, referred to as the $p$-Euler-alignment system. This model features a nonlinear velocity alignment force, interpreted as a density-weighted fractional $p$-Laplacian when the singularity parameter $\alpha$ exceeds the spatial dimension $d$. Our primary goal is to establish the existence of solutions for strongly singular interactions ($\alpha \ge d$) and compactly supported initial conditions. We construct solutions as mean-field limits of empirical measures from a kinetic variant of the $p$-Euler-alignment system. Specifically, we show that a sequence of empirical measures converges to a finite Radon measure, whose local density and velocity satisfy the $p$-Euler-alignment system. Our results are the first to prove the existence of solutions to this system in multi-dimensional settings without significant initial data restrictions, covering both nonlinear ($p>2$) and linear ($p=2$) cases. Additionally, we establish global existence, uniqueness, and collision avoidance for the corresponding particle ODE system under non-collisional initial conditions, extending previous results for $1 \le p \le \alpha + 2$. This analysis supports our mean-field limit argument and contributes to understanding alignment models with singular communication.
\end{abstract}

\section{Introduction}

Given $p\geq 2$, we consider the pressureless fractional Euler-alignment system with nonlinear velocity couplings:
\begin{equation}\label{e:macro}
    \left\{
        \begin{array}{ll}
            \displaystyle \partial_t \rho + \div_x(\rho u) = 0,\\
            \displaystyle \partial_t(\rho u) + \div_x (\rho u \otimes u ) = \int_{\R^d} \frac{ \left| u(y) - u(x)\right|^{p-2} (u(y)-u(x)) }{|x-y|^\alpha} \rho(y)\rho(x) \dd y,
        \end{array}
    \right.
\end{equation}
which we refer to as the $p$-Euler-alignment system. Here, the pair $(\rho(t,x), u(t,x))$ represents the density and velocity of agents occupying position $x\in \R^d$ at time $t\geq 0$. We use the notation $\rho_t(x) := \rho(t,x)$. The right-hand side of the momentum equation constitutes the nonlinear velocity alignment force, which in the case $\alpha>d$ can be interpreted as a $\rho$-weighted fractional $p$-Laplacian, $(-\Delta)^{(\alpha-d)/2}_p$. We refer to \cite{T-23} for general discussion on the system \eqref{e:macro}. 

Our objective is to demonstrate, under the strong singularity assumption
$$ \alpha\geq d, $$
that given a compactly supported initial density $\rho_0$ and a bounded initial velocity $u_0$, a solution $(\rho,u)$ (initiated at $(\rho_0,u_0)$) can be constructed as a mean-field limit of empirical measures. More precisely, there exists a sequence of empirical measures $\mu^N$, solving a kinetic variant of \eqref{e:macro}, which converges weakly to some finite Radon measure $\mu$, whose local density and local velocity
 \begin{equation}\label{loc_quant}
        \rho_t(x) := \int_{\R^d} \mu_t(x,\dd v) \quad \mbox{and} \quad u(t,x) := \frac{1}{\rho_t(x)}\int_{\R^d} v \dd \mu_t(x,\dd v),
    \end{equation}
    satisfy the $p$-Euler-alignment system \eqref{e:macro}. Denoting the space of probability measures as ${\mathcal P}(\R^d)$, the main result can be stated as follows.
\begin{theo}\label{t:EAexist}
    Let $T>0$, $\rho_0 \in \mathcal{P}(\R^d)$ be compactly supported and $u_0:\R^d\rightarrow \R^d$ belong to $L^\infty(\rho_0)$. Moreover assume that $\alpha\ge d$ and $2\leq p\in[\alpha,\alpha+2]$. Then there exists a family of compactly supported probability measures $\{\mu_t\}_{t \in [0,T]}$ such that its local quantities $(\rho, u)$, cf. \eqref{loc_quant},
    satisfy the following assertions:
    \begin{enumerate}
        \item If $p \in (\alpha, \alpha + 2]$, then $(\rho, u)$ constitute a weak solution to the $p$-Euler-alignment system \eqref{e:macro} in the sense of Definition \ref{d:weakEA};
        \item If $p=\alpha$, then $(\rho, u)$ constitute a weak solution to the $p$-Euler-alignment system \eqref{e:macro} in the sense of Definition \ref{d:weakEA}, provided that $\rho$ is non-atomic for a.a. times $t \in [0,T]$.
    \end{enumerate}
    Moreover, the measure $\mu$ can be constructed through a mean-field limit.
\end{theo}
The significance of the $p$-Euler-alignment system lies in its intriguing large-time behavior, recently studied by {\sc Tadmor} \cite{T-23} and {\sc Black and Tan} \cite{BT24, BT-arxiv23}. Notably, in \cite{BT24}, the authors show that for certain initial configurations, the model exhibits the so-called semi-unconditional flocking, which is a novel type of behavior for alignment models based on the Cucker-Smale framework. To the best of our knowledge, our work presents the first results on the existence of solutions to the $p$-Euler-alignment system \eqref{e:macro}. It is also among the initial contributions to existence results for Cucker-Smale-type alignment models with singular communication in multi-dimensions, achieved without significant restrictions on the initial data, even in the linear case of $p=2$.
 
 Establishing well-posedness, including existence, for the multidimensional fractional Euler-alignment system with $p=2$ remains one of the principal unsolved problems in alignment dynamics. Most results in multi-dimensional settings involve small data, see \cite{S-19, C-19, CTT-2021} for the multi-dimensional torus and \cite{DMPW-19} for Euclidean space. In the one-dimensional case, most results revolve around critical thresholds \cite{TT-14, T-20, HT-17, CCTT-16}. However, in the strongly singular case on the one-dimensional torus, these results have been refined by replacing critical thresholds with a less restrictive assumption of initial density being bounded away from vacuum \cite{DKRT-18, ST-17, ST-18, STII-17}. The impact of a vacuum is further analyzed in \cite{AC-21, T-20}. The challenge of addressing the multi-dimensional case for general initial data remains open.

 The recent work \cite{FP-23} provides the existence of solutions in one and two dimensions without imposing significant assumptions on the initial data, as in Theorem \ref{t:EAexist}, but specifically for $p=2$. A crucial ingredient in their analysis is the uniform boundedness of the integral
 \begin{equation}\label{p=2_ass}
     \int_{\R^{2d}}|x-x'|^{2 - \alpha} \dd \rho_t(x) \dd \rho_t(x'),
 \end{equation}
 which, for compactly supported $\rho_t$ and $\alpha \le 2$, is immediate. The latter assumption, together with the strong singularity $\alpha\ge d$, limited the $d$ to either $1$ or $2$. The ''no atoms'' assumption in the critical case of $\alpha=p$ has a similar origin to the concept of diagonal defect measures in the works of Schochet e.g. \cite{S-96} and the recent work \cite{PR-24}.
 
 Adequate control of \eqref{p=2_ass} with $\alpha\geq d>2$ appears to be crucial for proving the existence of solutions for the multi-dimensional Euler-aligmnent system, however it remains open. Meanwhile, by introducing the nonlinear velocity coupling in \eqref{e:macro} we can circumvent the problem since then, the natural counterpart of \eqref{p=2_ass} is
  \begin{equation}\label{p>2_ass}
     \int_{\R^{2d}}|x-x'|^{p - \alpha} \dd \rho_t(x) \dd \rho_t(x'), 
 \end{equation}
 whose boundedness requires only that $\alpha\leq p$. With $p\geq d$ this 
 allows us to extend the available range of $\alpha$ and deal with the multi-dimensional problem. Remarkably, the nonlinear character of the velocity coupling in \eqref{e:macro} causes no problem in our method, since we treat \eqref{e:macro} as a monokinetic reduction of its kinetic counterpart wherein everything is linear, see Section \ref{sec:hydromono}.

 We construct the candidate for the solution $\{\mu_t\}_{t\in[0,T]}$ in Theorem \ref{t:EAexist} by approximation with atomic solutions in the form of empirical measures
 \begin{equation}\label{e:empirical}
    \mu_t^N (x,v) := \frac{1}{N}\sum\limits_{i=1}^N \delta_{x_i^N(t)}(x) \otimes \delta_{v_i^N(t)}(v) \in \mathcal{P}(\R^{2d}).
\end{equation}
Here, $(x_i^N, v^N_i)_{i=1}^N$ is a solution to the $p$-Cucker-Smale ($p$-CS) system
\begin{equation}\label{e:micro}
\left\{
    \begin{array}{ll}
    \dot{x}_i = v_i, \\
    \dot{v}_i = \displaystyle\frac{1}{N}\displaystyle\sum\limits_{i\neq j=1}^N  \frac{|v_j-v_i|^{p-2}(v_j-v_i)}{|x_i-x_j|^\alpha},
    \end{array}
\right. \qquad (x_i(0), v_i(0))= (x_{i0}, v_{i0}) \in \R^{2d},
\end{equation}
where $(x_i(t), v_i(t))$ is the position and velocity of $i$th agent at the time $t\geq 0$. Thus, the need arises to study well-posedness for \eqref{e:micro}. General nonlinear velocity coupling $\Gamma(v_i - v_j)$ has been first proposed in \cite{HHK-10}. Nonlinear velocity coupling as in \eqref{e:micro}, coupled with singular communication weight $\psi(s) = s^{-\alpha}$ has been studied by {\sc Markou} in \cite{M-18}, where the main contribution was collision-avoidance and well-posedness for $\alpha \ge 1$ and $1 \le p \le 3$. However, as evident in \eqref{p>2_ass}, well-posedness of \eqref{e:micro} in the case of $p> d$ is essential from our point of view, and we require an improvement to Markou's result. It constitutes our second main result. 

% \begin{theo}\label{t:colavoid_intro}
% Suppose that $\alpha\geq 1$ and $1\leq p\leq\alpha+2$.
% Moreover, let the  initial data $(x_{i0},v_{i0})_{i=1}^N$ be non-collisional, i.e.
% $$ x_{i0}\neq x_{j0},\quad \mbox{ for all }\ i,j\in\{1,...,N\}. $$

% \noindent
% Then there exist a classical, global, unique solution $(x_{i}(t),v_{i}(t))_{i=1}^N$ to \eqref{e:micro}, which remains non-collisional indefinitely. On the other hand, for $p>\alpha+2$ and $\alpha>1$, system \eqref{e:micro} allows collisions in finite time, see Remark \ref{rem:collision}.
% \end{theo}

\begin{theo}\label{t:colavoid_intro}
Let $\alpha \geq 1$ and $p \geq 1$. Suppose the initial data $(x_{i0},v_{i0})_{i=1}^N$ are non-collisional, i.e.
$$ x_{i0}\neq x_{j0},\quad \mbox{ for all }\ i \neq j\in\{1,...,N\}. $$

\noindent
Then, for $p \leq \alpha+2$, there exist a classical, global, unique solution $(x_{i}(t),v_{i}(t))_{i=1}^N$ to \eqref{e:micro}, which remains non-collisional indefinitely. 
\end{theo}

\begin{rem}\rm
    The range of $p$ for the non-collisional behavior of solutions stated in Theorem \ref{t:colavoid_intro} is optimal, i.e. for $p > \alpha+2$ system \eqref{e:micro} allows collisions in finite time, see Proposition \ref{rem:collision} below for details.  
\end{rem}

\subsection{Cucker-Smale model and its kinetic limit.}  
Let us discuss the linear case $p=2$, which brings us back to the classical Cucker-Smale model in \eqref{e:micro}. Since its introduction by {\sc Cucker and Smale} in 2007 \cite{CS-07}, this model has been extensively studied from various perspectives \cite{CCP-17, CFRT-10, CHL-17, CKPP-19, HL-09, HT-08, MMPZ-19, S-2021-Surv}. Of particular interest in this paper 
is the case of singular communication weight $\psi(s)=s^{-\alpha}$. This specific scenario was introduced by {\sc Ha and Liu} in \cite{HL-09}, where they investigated its large-time behavior. The existence theory for $\alpha\in(0,1)$ was subsequently developed by the third author in \cite{P-14, P-15}. Later, well-posedness and collision avoidance in the strongly singular case of $\alpha\geq 1$ were established in \cite{CCMP-17}. The mean-field limit, transitioning from microscopic to kinetic descriptions for $\alpha\in\left(0,\frac{1}{2}\right)$, was addressed in \cite{MP-18}, with further advancements in one dimension presented in \cite{CZ-21, CC-21}. The local-in-time existence theory of weak solutions for the kinetic Cuker-Smale equation was developed in \cite{CCH-14} for $\alpha \in (0,d-1)$ and extended to $\alpha \in (0,d)$ in \cite{CJ-24}. Recently, the first author, together with {\sc Jung}, demonstrated the existence of solutions to a relaxation of the kinetic Cucker-Smale equation in multi-dimensions with super-Coulombian singularity \cite{CJ-23}. Concurrently, the last two authors showed that, without such a relaxation, the kinetic Cucker-Smale equation enforces monokineticity and can be reduced to a weak solution of the Euler-alignment system \cite{FP-23}. These results are closely related to fundamental issues in the mean-field limit for particle systems with Coulombian interactions, an area which has recently seen significant contributions \cite{BJS-23, RS-23, NRS-22, L-23}. Further research directions include exploring various types of communication and scaling in the Cucker-Smale model \cite{MT-11, MMP-20, ST-top-20}, with {\sc Shvydkoy}'s novel concept of environmental averaging \cite{S-22} being particularly noteworthy.

For comprehensive overviews, we recommend the surveys \cite{CCP-17, CHL-17, MMPZ-19, S-2021-Surv} and the references therein.

\subsection{Hydrodynamic limit and the role of monokineticity}\label{sec:hydromono}
Formally, macroscopic equations \eqref{e:macro} can be obtained by taking the first two $v-$moments of the kinetic equation
\begin{equation}\label{e:meso}
    \left\{
      \begin{array}{ll}
        \displaystyle \partial_t \mu + v \cdot \nabla_x \mu + \div_v \left(F(\mu)\mu\right) = 0, \quad x\in \R^d, \; v \in \R^d  \\
        \displaystyle F(\mu)(t,x,v) := \int_{\R^{2d}} \frac{|w-v|^{p-2}(w-v)}{|x-y|^\alpha} \mu(t,y,w) \dd w \dd y,
      \end{array}  
    \right.
\end{equation}
where $\mu = \mu(t,x,v)$ is the distribution of particles passing through $(x,v) \in \R^{2d}$ at time $t \ge 0$. Thus, integrating the equation \eqref{e:meso} in $v$ and exploiting the skew-symmetry of $F(\mu)$, we obtain the continuity equation \eqref{e:macro}$_1$, for the local quantities \eqref{loc_quant}.

To obtain the momentum equation, we take the first $v-$moment of kinetic equation $\eqref{e:meso}$. Formal derivation of the Euler-alignment system can be found in \cite{HT-08, C-19}, while for rigorous limit from the kinetic to the hydrodynamical description we recommend \cite{FK-19} (pressureless limit) and \cite{KMT-15} (Euler-alignment with pressure).
Direct micro-to-macroscopic mean-field limit was performed by the first author and {\sc Carrillo} in \cite{CC-21}.

While computing the first local velocity moment of \eqref{e:meso} we encounter two problems. The first one is related to the transport term $v\cdot \nabla_x \mu$. Multiplying it by $v$ and integrating, we have
\begin{equation*}
    \int_{\R^d} v \otimes v \mu \dd v = \rho u \otimes u + \mathbb{P},
\end{equation*}
where the pressure term $\mathbb{P}$, still kinetic and dependent on $v$, reads
\begin{equation*}
    \mathbb{P} = \int_{\R^d} (v-u) \otimes (v-u) \mu \dd v.
\end{equation*}
Secondly, the right-hand side in \eqref{e:macro}$_2$ becomes highly nonlinear and kinetic:
\begin{equation*}
    \rho|w-u|^{p-2}(w-u) = \int_{\R^d} |w-v|^{p-2}(w-v) \mu(t,x,v) \dd v.
\end{equation*}
Therefore we arrive at the following situation. Firstly, while dealing with the singularity of the communication weight, it is more convenient to work with the hydrodynamical description of \eqref{e:macro}, since then the local velocity $u$ better counteracts the singularity. Secondly, on the hydrodynamical level three additional problems appear: the emergence of nonlinear convective term and nonlinear velocity couplings (both of which are linear on the kinetic level!) as well the justification of the monokinetic ansatz, required to ensure that ${\mathbb P}=0$. All of these problems become manageable as soon as we can show that the strong singularity forces monokineticity, as we do using Theorem \ref{main1} (see preliminaries below).
Using notation developed in Section \ref{sec:monokin} we prove that the candidate for a solution to \eqref{e:meso} is of the monokinetic form $\mu(t,x,v) = \delta_{u(t,x)}(v) \otimes \rho_t(x) \otimes \lambda^1(t)$, which makes it both kinetic and hydrodynamical at the same time. Thus, we can simultaneously deal with hydrodynamical nonlinearity at the kinetic level and with problems related to the singularity of communication at the hydrodynamical level.

The rest of the paper is organized as follows. In Section \ref{sec:prelim}, we introduce the necessary preliminaries, including weak formulations and general results related to monokineticity. We employ the following strategy. We aim to prove that atomic solutions \eqref{e:empirical} exist and converge, up to a subsequence to a candidate $\mu$ for a solution of \eqref{e:meso}. Then using both good properties of \eqref{e:meso} and \eqref{e:macro} we show that even though $\mu$ does not necessarily solve \eqref{e:meso}, its local quantities \eqref{loc_quant} constitute a solution to \eqref{e:macro}. Thus our first step in Section \ref{sec:colavoid} is the proof of Theorem \ref{t:colavoid_intro} and other information regarding the particle system, ensuring the existence of atomic solutions. In Section \ref{sec:compact}, we prove intermediate results regarding the compactness of the atomic solutions, which play a crucial role in the mean-field limit. In Section \ref{sec:mflimit}, we show that any limit measure obtained by the mean-field method is monokinetic and its local quantities satisfy $p-$Euler-alignment system, which constitutes a proof of Theorem \ref{t:EAexist}. Some technical results are relegated to the Appendix.

\subsection*{Notation}
Throughout the paper ${\mathcal P}(\R^d)$, ${\mathcal M}(\R^d)$ and ${\mathcal M}_+(\R^d)$ denote the space of probability measures, finite Radon measures and nonnegative finite Radon measures on $\R^d$, respectively. Given $\mu, \nu \in \mathcal{M}_+(\R^d)$, the bounded-Lipschitz distance $d_{BL}(\mu, \nu)$ between them is defined as
$$
d_{BL}(\mu,\nu) := \sup\limits_{\phi \in \Gamma_{BL}} \left| \int_{\R^d} \phi \dd \mu - \int_{\R^d} \phi \dd \nu \right|,
$$
where $\Gamma_{BL}$ is the set of admissible test functions
$$
\Gamma_{BL} := \left\{ \phi: \R^d \longmapsto \R : ||\phi||_{\infty} \le 1,\ [\phi]_{Lip}:= \sup\limits_{x\neq y}\frac{|\phi(x)-\phi(y)|}{|x-y|}\le 1 \right\}.
$$
Since all the measures considered throughout the paper are uniformly compactly supported, it is worthwhile to note that, for such measures, the bounded-Lipschitz distance induces topology equivalent to the narrow topology, characterised by convergence tested with bounded-continuous functions.

We use the notation 
$\R^d_x$   and  $\R^d_v$
to emphasize that these are the spaces of position $x$ and velocity $v$, respectively. To describe disintegrations of measures, such as
$ \mu(t,x,v) = \delta_{u(t,x)}(v) \otimes \rho_t(x) \otimes \lambda^1(t)$ or $[\mu_t\otimes \mu_t]\otimes\lambda^1(t)$ we use the notation from Section \ref{sec:monokin} and particularly formula \eqref{disint_form}. 

For a Borel-measurable mapping $F:\R^d \rightarrow \R^n$ and a Radon measure $\mu$ we define the \textit{pushforward measure} $F_{\#}\mu$ by
    \begin{equation}\label{pushforward}
    F_{\#}(\mu)(B):= \mu\left( F^{-1}(B) \right) \quad \text{for Borel sets } B.
    \end{equation}
    The main property of the pushforward measure,  used throughout the paper, is the change of variables formula
    \begin{equation}\label{changeofvariables}
        \int_{\R^n} g \dd (F_{\#}\mu) = \int_{\R^d} g \circ F \dd \mu,
    \end{equation}
    whenever $g\circ F$ is $\mu$-integrable.

Finally, throughout the paper, we use the notation $A\lesssim B$ to state that there exists a ''harmless'' positive constant $C$, such that $A\leq C\ B$.

\section{Preliminaries}\label{sec:prelim}
In this section, we introduce the main mathematical tools utilised in the paper, including weak formulations for meso- and macroscopic systems.

\subsection{Weak formulations}
We begin by introducing the kinetic weak formulation.

% Before proceeding with the weak formulation for the kinetic CS equation \eqref{e:meso}, we need to introduce the following objects:
% \begin{itemize}
%     \item \textbf{Diagonal set} $\Delta$ defined as
%     \begin{equation}\label{diagonal}
%     \Delta:= \{(x,v,x',v')\in\R^{4d}:\quad x=x'\}\quad\mbox{or}\quad \Delta:= \{(x,x')\in\R^{2d}:\quad x=x'\},
% \end{equation}
% depending on the context;
%     \item \textbf{Kinetic energy} $E[\mu_t]$ defined as
%     $$
%     E[\mu_t] := \int_{\R^{2d}} |v|^2 \dd \mu_t(x,v);
%     $$
%     % \item \textbf{$p$-energy} $E_p[\mu_t]$, for $p>2$, defined as 
%     % $$
%     % E_p[\mu_t] := \int_{\R^{2d}} |v|^p \dd \mu_t(x,v);
%     % $$
%     % {\color{red} Póki co nie wiem, czy to się przyda.}
%     \item \textbf{Energy dissipation rate (or enstrophy)} $D_p[\mu_t]$ defined as 
%     $$
%     D_p[\mu_t] := \int_{\R^{4d}\setminus\Delta} \frac{|v-v'|^p}{|x-x'|^{\alpha}} \dd\left[\mu_t(x,v) \otimes \mu_t(x',v')\right].
%     $$
% \end{itemize}
% With those notions at hand, we introduce the weak formulation.
\begin{definition}[\textbf{Weak solution of the kinetic equation}]\label{d:weakkin}
For a fixed $0<T<\infty$ we say that $\mu = \{\mu_t\}_{t\in[0,T]} \in C\left( [0,T]; (\mathcal{P}(\R^{2d}), d_{BL}) \right)$ is a weak measure-valued solution of \eqref{e:meso} initiated in $\mu_0 \in \mathcal{P}\left(\R^{2d}\right)$ if the following assertions are satisfied.
\begin{enumerate}[label=(\roman*)]
    \item There exists a constant $M>0$ such that for a.a. $t\in[0,T]$ we have
    $$\text{spt}(\mu_t)\subset\subset (T+1)B(M)\times B(M)\subset \R^d_x\times\R^d_v.$$
  In other words, the $v$-support of $\mu_t$ is compactly contained in the ball $B(M)$, while the $x$-support of $\mu_t$ is compactly contained in $(T+1)B(M)$.
    \item For each $\phi \in C^1([0,T]\times\R^{2d})$, compactly supported in $[0,T)$ with Lipschitz continuous $\nabla_v\phi$ in $B(M)$, the following identity holds
            \begin{equation}\label{e:weakkin}
            \begin{split}
            &- \int_{\R^{2d}}\phi(0,x,v) \dd \mu_0(x,v) = \int_0^T \int_{\R^{2d}} \left(\partial_t \phi(t,x,v) + v\cdot \nabla_x \phi(t,x,v) \right) \dd \mu_t(x,v) \dd t\\
            &+ \frac{1}{2} \int_0^T \int_{\R^{4d}\setminus\Delta} \frac{\left(\nabla_v \phi(t,x,v) - \nabla_v \phi(t,x',v')\right)\cdot (v - v')|v-v'|^{p-2}}{|x-x'|^{\alpha}} \dd [\mu_t(x,v) \otimes \mu_t(x',v')] \dd t,
            \end{split}
            \end{equation}
            where $\Delta = \{(x,v,x',v')\in\R^{4d}:\ x = x'\}$ is the diagonal set.
            In particular, all integrals in the above equation are well-defined.
\end{enumerate}
\end{definition}

\begin{rem}\rm\label{r:wkin_prop}
    We treat the weak formulation above mostly as a stepping stone to the macroscopic scale, which we describe carefully in the next subsection.   Nevertheless, the following comments are in order. They follow, mutatis mutandis, similar statements for the $p=2$ case in \cite{FP-23}.  Assume that for a fixed $0<T<\infty$ we have $\mu_0$ and $\mu$ as above. Then:
    
    \begin{enumerate}[label = \roman*)]
    \item Solutions $\mu$ in the sense of Definition \ref{d:weakkin} satisfy the energy equality
        \begin{equation}\label{kin_e_eq}
        \int_0^t D_p[\mu_s] \dd s = E[\mu_0] - E[\mu_t]
        \end{equation}
        for a.a. $t \in [0,T]$, where
        \begin{equation}\label{kinen}
        \begin{split}
            E[\mu_t] &:= \int_{\R^{2d}} |v|^2 \dd \mu_t(x,v),\\
            D_p[\mu_t] &:= \int_{\R^{4d}\setminus\Delta} \frac{|v-v'|^p}{|x-x'|^{\alpha}} \dd\left[\mu_t(x,v) \otimes \mu_t(x',v')\right],
        \end{split}
        \end{equation}
        are the kinetic energy and its dissipation rate (sometimes called enstrophy), respectively.
        Conversely, if \eqref{kin_e_eq} is satisfied, then all the terms in \eqref{e:weakkin} are well-defined. At the macroscopic level, this equality translates into inequality and no longer is a consequence of a weak formulation, therefore we include it in Definition \ref{d:weakEA}.

        \item For $\alpha \ge d$, solutions $\mu$ in the sense of Definition \ref{d:weakkin} are monokinetic and their local quantities \eqref{loc_quant} satisfy the $p$-Euler-alignment system in the weak sense of Definition \ref{d:weakEA}.
\end{enumerate}

\end{rem}

Next, we introduce the main tool linking the particle and the kinetic level.

\begin{definition}[\textbf{Atomic solution}]\label{d:empirical}
For a fixed $N\in\mathbb{N}$, let $(x_i^N(t), v_i^N(t))_{i=1}^N$ be a solution to the particle system \eqref{e:micro}. We define the atomic solution $\mu^N\in{\mathcal M}([0,T]\times\R^{2d})$ associated with $(x_i^N(t), v_i^N(t))_{i=1}^{N}$ as the family of empirical measures $\{\mu_t^N\}_{t\in[0,T]}$ (cf. \eqref{mudis}) of the form 
\begin{equation*}%\label{atomsol}
    \mu_t^N(x,v) := \frac{1}{N} \sum\limits_{i=1}^N \delta_{x_i^N(t)}(x) \otimes \delta_{v_i^N(t)}(v),\qquad \mbox{for all }t\in[0,T].
\end{equation*}
\end{definition}

\begin{rem}\rm
    For each $N$, the atomic solution $\mu_t^N$ is a solution to the kinetic equation \eqref{e:meso} in the sense of Definition \ref{d:weakkin}. Indeed, plugging $\mu^N_t$ into Definition \ref{d:weakkin} and integrating by parts in $t$, the equation \eqref{e:weakkin} translates into the particle equations \eqref{e:micro}; see for example \cite{MP-18}.
\end{rem}

Finally, we focus on the weak formulation for the $p$-Euler-alignment system \eqref{e:macro}. We interpret it as a monokinetic reduction of the weak formulation in Definition \ref{d:weakkin}.

\begin{definition}[\textbf{Weak solution to the $p-$Euler-alignment system}]\label{d:weakEA} For a fixed $0<T<\infty$ we say that the pair $(\rho,u)$ with $\rho_t \in C([0,T]; \left(\mathcal{P}(\R^d), d_{BL})\right)$ and $u\in L^1([0,T]; L^1(\rho_t)) \cap L^\infty(\rho)$ is a weak solution of \eqref{e:macro} with compactly supported initial data $\rho_0 \in \mathcal{P}(\R^d)$ and $u_0\in L^\infty(\rho_0)$ if all the following assertions are satisfied.
\begin{enumerate}[label=(\roman*)]

\item There exists a constant $M>0$ such that for a.a. $t \in [0,T]$ we have 
\begin{equation*}
    \text{spt}(\rho_t) \subset \subset (T+1)B(M) \subset \R^d_x \quad \text{and} \quad \lVert u(t,\cdot) \rVert_{L^\infty(\rho_t)}<M;
\end{equation*}

\item For a.a. $t\in [0,T]$
\begin{equation*}
    \int_0^t \int_{\R^{2d}\setminus \Delta}\frac{|u(t,x)-u(t,x')|^p}{|x-x'|^\alpha} \dd \left[ \rho_s(x) \otimes \rho_s(x')\right] \dd s \le \int_{\R^d} |u(0,x)|^2 \dd \rho_0(x) - \int_{\R^d} |u(t,x)|^2 \dd \rho_t(x);
\end{equation*}

\item For each $\phi \in C^1([0,T]\times \R^d)$ and each $\phi_d \in C^1([0,T]\times \R^d; \R^d)$, compactly supported in $[0,T)$, the following equations are satisfied
\begin{equation}\label{e:weakEA}
    \begin{aligned}
        \int_0^T &\int_{\R^d} \left( \partial_t \phi + \nabla_x \phi \cdot u \right) \dd \rho_t(x) \dd t = -\int_{\R^d} \phi(0,x) \dd \rho_0(x) \\ 
        &\int_{\R^d} u(0,x) \cdot \phi_d(0,x) \dd \rho_0(x) 
        + \int_0^T \int_{\R^d} \left( \partial_t \phi_d \cdot u + u\left(u\cdot \nabla_x \right) \phi_d \right) \dd \rho_t \dd t \\
        = -\frac{1}{2}\int_0^T &\int_{\R^{2d}\setminus\Delta} \frac{( \phi_d(t,x) - \phi_d(t,x'))\cdot (u(t,x) - u(t,x'))|u(t,x)-u(t,x')|^{p-2}}{|x-x'|^\alpha} \dd [\rho_t \otimes \rho_t] \dd t.
    \end{aligned}
\end{equation}
\end{enumerate}
    
\end{definition}

\bigskip

\subsection{Monokineticity}\label{sec:monokin}
Next, we provide a rigorous introduction to the notion of monokineticity as well as primarily related tools. Before proceeding, we recall the disintegration theorem, see \cite{AGS-08} for details.
\begin{theo}[\textbf{Disintegration theorem}]\label{disintegration}
For $d_1, d_2\in\mathbb{N}$ denote projection onto the second componenent as $\pi_2:\R^{d_1}\times \R^{d_2} \longrightarrow \R^{d_2}$. For $\mu \in \mathcal{P} \left( \R^{d_1} \times \R^{d_2}\right)$, define its projection onto second factor as $\nu :=(\pi_2)_{\#} \mu \in \mathcal{P}(\R^{d_2})$. Then there exists a family of probabilistic measures $\left\{ \mu_{x_2}\right\}_{x_2 \in \R^{d_2}}\subset \mathcal{P}(\R^{d_1})$, defined uniquely $\nu$ almost everywhere, such that 
\begin{enumerate}[label=(\roman*)] 
    \item The map
    $$
    \R^{d_2} \ni x_2 \longmapsto \mu_{x_2}(B)
    $$
    is Borel-measurable for each Borel set $B \subset \R^{d_1}$;
    \item The following formula 
    $$
    \int_{\R^{d_1}\times\R^{d_2}} \phi(x_1, x_2) \dd \mu(x_1,x_2) = \int_{\R^{d_2}}\left( \int_{\R^{d_1}} \phi(x_1, x_2) \dd \mu_{x_2}(x_1) \right) \dd \nu (x_2)
    $$
    holds for every Borel-measurable map $\phi:\R^{d_1} \times \R^{d_2} \longmapsto [0, \infty)$.
\end{enumerate}
\end{theo}

The family $\{\mu_{x_2} \} $ is called a \textit{disintegration} of $\mu$ with respect to marginal distribution $\nu$. For the sake of simplicity, we often refer to formula \textit{(ii)} above as 
\begin{equation}\label{disint_form}
   \mu(x_1, x_2) = \mu_{x_2}(x_1) \otimes \nu(x_2). 
\end{equation}

Assume $\mu = \mu(t,x,v)$ is a finite Radon measure on $[0,T]\times \R^{2d}$ and that its $t$-marginal is the 1D Lebesgue measure on $[0,T]$, which we denote by $\lambda^1(t)$.  By Theorem \ref{disintegration}, for a.a. $t \in [0,T]$ there exists a measure $\mu_t(x,v) \in \mathcal{P}(\R^{2d})$ such that 
\begin{equation}\label{mudisbis}
\mu(t,x,v) = \mu_t(x,v) \otimes \lambda^1(t)\qquad \mbox{(see notation \eqref{disint_form})}.
\end{equation}
Disintegrating $\mu_t$ further with respect to its local density $\rho_t(x)$ (cf \eqref{loc_quant}), we obtain
\begin{equation}\label{mudis}
\mu(t,x,v) = \sigma_{t,x}(v) \otimes \rho_t(x) \otimes \lambda^1(t),
\end{equation}
where the measure $\sigma_{t,x} \in \mathcal{P}(\R^d_v)$, governing the local velocity distribution of $\mu$, is defined for a.a. $t \in [0,T]$ and $\rho_t$-a.a. $x$. Throughout the paper, we identify $\mu$ with the family $\{\mu_t\}_{t \in [0,T]}$, understood as above. We say that $\mu$ is \textit{monokinetic} if for $a.a.$ $t \in [0,T]$ and $\rho_t-$a.a. $x$ the measure $\sigma_{t,x}$ is a Dirac delta concentrated in some $u(t,x) \in \R^d$, i.e.
\begin{equation}\label{mumono}
    \mu(t,x,v) = \delta_{u(t,x)}(v) \otimes \rho_t(x) \otimes \lambda^1(t).
\end{equation}
Then $(t,x) \mapsto u(t,x)$, defined for a.a. $t \in [0,T]$ and $\rho_t$-a.a. $x$, plays the role of velocity.

The main tool that allows us to connect the kinetic and hydrodynamical levels is \cite[Theorem 1.1]{FP-23}, which ensures the monokineticity of the measure $\mu$ under the mild assumptions which we summarise in the following definition.
% It was motivated by the results in [xxJabinReyxx], where the authors proved that weak, measure-valued solutions to the equation
% $$
% \partial_t \mu + v \partial_x \mu = - \partial_{vv}m
% $$
% (with $m$ a given measure) are monokinetic. The crucial assumption was that $\mu$ is BV with values in some negative space. Such a framework is already too restrictive to apply to the singular CS model \eqref{e:micro} and \eqref{e:meso}, since our version of $\mu$ is defined a.e. and as such can have discontinuities. Therefore, still following [xxFPxx], we introduce two weaker conditions that are sufficient for $\mu$ to be monokinetic. 

\begin{definition}\label{d:MP+SF}
Assume that the measure $\mu$ satisfies \eqref{mudisbis} and \eqref{mudis} and that $\mu_t$ is compactly supported in $B(M)(T+1) \times B(M)$ for a.a. $t \in [0,T]$, where $B(M)$ denotes a ball of radius $M>0$ centered at $0$. \\
We say that $\mu$ is:

\begin{enumerate}[label=(\alph*)]
    \item \underline{Locally Mass Preserving} (MP) if for any bounded set $C\subset \R^d$ and any $t_0 \le t$ 
    \begin{equation}\label{MP}
        \rho_t \left(\overbar{C} + (t-t_0) \overbar{B(M)}\right) \ge \rho_{t_0}(C);
    \end{equation}

    \item \underline{Steadily Flowing} (SF) if there exists a full measure set $A \subset [0,T]$, such that for all $t_0 \in A$ and all $0 \le \phi \in C^\infty_c (\R^d)$ the family of finite Radon measures $\rho_{t_0,t}[\phi]$ defined for any Borel set $B \subset \R^d$ as 
    \begin{equation}\label{SF}
        \rho_{t_0, t}[\phi](B) := \int_{B \times \R^d} \phi(v) \dd T^{t_0,t}_\# \mu_t
    \end{equation}
    is narrowly continuous at $t=t_0$, as a function of $t \in [0,T]$, restricted to $A$. Here $T^{t_0, t}_\# \mu_t$ is the pushforward measure defined by the mapping $T^{t_0, t}(x,v) = (x - (t-t_0)v,v)$, see notation in \eqref{pushforward} and \eqref{changeofvariables}.
\end{enumerate}
\end{definition}

\begin{theo}\label{main1}
Suppose that the family of uniformly compactly supported probability measures $\{\mu_t\}_{t\in[0,T]}$ on $\R^{2d}$
is locally mass-preserving and steadily flowing. Assume further that
\begin{align}\label{main1-as1}
    \int_0^T D^\alpha[\mu_t]\dd t:= \int_0^T \int_{\R^{4d}\setminus \{x=x'\}} \frac{|v-v'|^{\alpha+2}}{|x-x'|^\alpha}\dd [\mu_t\otimes\mu_t] \dd t <\infty, \qquad \alpha\geq d.
\end{align}
Then $\mu$ is monokinetic, i.e. it admits the disintegration \eqref{mumono}.
\end{theo}

\begin{rem}\rm
      Definition \ref{d:MP+SF} is motivated by \cite{JR-16}, wherein the authors assume the measure to be BV in time. Our weaker assumptions (MP) and (SF) are tailored to replace this time regularity and control the manner in which $\mu$ dissipates mass. The pushforward measure $T^{t_0, t}_\# \mu_t$ introduced in (SF) transforms only the position, leaving velocity intact, and serves as the only concept of the flow that we can reasonably use. It needs to be introduced for the following reason. We aim to show that the limit of atomic solutions is monokinetic, and, unlike \cite{JR-16}, we do not a priori know that such limit solves any equation, and thus even the relation between $x$ and $v$ as ''position'' and ''velocity'' is unclear. Taking the pushforward with respect to $(x-(t-t_0) v)$ helps to recognize $v$ as the velocity of the motion of position $x$. 
     At the same time, condition (MP) prevents loss of mass along such a flow. We point out that if $t \mapsto \mu_t$ is narrowly continuous or BV in time, both conditions are satisfied.
 \end{rem}

\section{Collision-avoidance}\label{sec:colavoid}
This section is dedicated to the well-posedness and collision-avoidance for system \eqref{e:micro}. It is the first step of our efforts since it ensures the existence of the approximation by atomic solutions. However, we begin by providing additional information such as growth conditions on the maximum position and velocity of the particles, as well as energy equality for the system.
\begin{prop}[Propagation of position and velocity]\label{p:flock}
Let $(x_i(t), v_i(t))_{i=1}^N$ be a classical solution to \eqref{e:micro} in $[0,T]$. Then we have 
\begin{equation*}
    \sup_{t\geq 0}\max\limits_{1\le i \le N} |v_i(t)| \le \max\limits_{1\le i \le N} |v_{i0}|, \quad 
    \sup_{t\geq 0}\max\limits_{1\le i \le N} |x_i(t)| \le \max\limits_{1\le i \le N} |x_{i_0}| + t\max\limits_{1\le i \le N} |v_{i0}|.
\end{equation*}
\end{prop}

\noindent
The proof follows the same idea as in the case of a standard CS model. We sketch it here for the reader's convenience.
\begin{proof}
     The bound on the position follows immediately from the first equation in \eqref{e:micro} and from the bound on the velocity. Thus, we focus on the bound on the velocity. Denote $V(t):=\max\limits_{1\le i \le N} |v_i(t)|$ and note that it is an absolutely continuous function as a maximum of finitely many absolutely continuous functions. In particular, it is differentiable a.e., and its derivative satisfies the fundamental theorem of calculus. Fix a point $t_0$ of differentiability of $V$ and suppose that $K$ is the set of all indices $i$ such that at $t=t_0$ we have $V(t)=|v_i(t)|$. Then $\frac{\dd}{\dd t}V(t_0)\leq \max_{i\in K}\frac{d}{\dd t}|v_i(t_0)|$ and for all $i\in K$,
    \begin{align*}
        \frac{\dd}{\dd t}|v_i|^2 &= v_i\cdot\frac{1}{N}\sum\limits_{i\neq j=1}^N \frac{|v_j-v_i|^{p-2}}{|x_i-x_j|^\alpha}(v_j-v_i)\\
        & = \frac{1}{N}\sum\limits_{i\neq j=1}^N  \frac{|v_j-v_i|^{p-2}}{|x_i-x_j|^\alpha}(v_j\cdot v_i-|v_i|^2)\leq 0,\quad \mbox{at }t=t_0,
    \end{align*}
    where the last inequality follows from the fact that every involved factor is nonnegative except 
    $$v_j(t_0)\cdot v_i(t_0) - |v_i(t_0)|^2\leq V^2(t_0)-V^2(t_0)= 0.$$
    Thus, $\frac{\dd}{\dd t}V(t_0)\leq 0$ in all points of differentiability of $V$. Since $V$ is absolutely continuous, it implies that it is non-increasing.
\end{proof}
\noindent The above proposition implies the existence of a positive constant $M$ such that

\begin{equation}\label{e:boundM}
    \sup_{t\geq 0}\sup\limits_{1\le i \le N} |v_i(t)| < M,\quad \sup_{t\geq 0}\sup\limits_{1\le i \le N} |x_i(t)|< (T+1)M.
\end{equation}

\begin{prop}[Energy equality]\label{p:eneq}
 Let $(x_i(t), v_i(t))_{i=1}^N$ be a classical solution to \eqref{e:micro} in $[0,T]$. Then for all $t\in[0,T]$, we have
 \begin{equation*}
     \frac{1}{2}\frac{1}{N^2}\int_0^t\sum_{i\neq j = 1}^N|v_i(s)-v_j(s)|^p|x_i(s)-x_j(s)|^{-\alpha} \dd s = \frac{1}{N}\sum_{i=1}^N|v_i(0)|^2 - \frac{1}{N}\sum_{i=1}^N|v_i(t)|^2 .
 \end{equation*}
\end{prop}
\begin{proof}
    The proof is the same as in the classical case. One simply needs to calculate the time derivative of the quadratic energy $t\mapsto\frac{1}{N}\sum_{i=1}^N|v_i(t)|^2$ and use the standard trick of exchanging the indices $i$ and $j$ in the double sum. We omit the details.
\end{proof}

%%%%%%%%%%%%%%% ZOSTAWIAM NA WYPADEK MODELU Z INNĄ NIELINIOWOŚCIĄ ctrl+f: xxinnaXX %%%%%%%%%
 % The proof, similar to Proposition \ref{p:flock} follows along the lines of a classical argument for the CS model. However, it reveals some particularities specific to the $p$-CS system, and thus, we shall present it.
% \begin{proof}
%    We have
%    \begin{align*}
%           \frac{1}{2}\frac{d}{dt}\frac{1}{N}\sum_{i=1}^N|v_i|^2 &= \frac{1}{N}\sum_{i=1}^Nv_i\cdot \frac{1}{N}\sum\limits_{i\neq j=1}^N \psi(|x_i-x_j|) |v_j-v_i|^{p-2}(v_j-v_i) \\
%           & = \frac{1}{N^2}\sum_{i\neq j=1}^N\psi(|x_i-x_j|) |v_j-v_i|^{p-2}(v_j-v_i)\cdot v_i.
%    \end{align*}
%     Using the classical trick with exchanging the indices $i$ and $j$ in the double sum, we arrive at
%     $$ \frac{1}{2}\frac{d}{dt}\frac{1}{N}\sum_{i=1}^N|v_i|^2 =  $$
% \end{proof}

Finally, we prove the main result related to the $p$-CS particle system \eqref{e:micro}, Theorem \ref{t:colavoid_intro}, which combines well-posedness and collision-avoidance.

\smallskip
\noindent

\begin{proof}[Proof of Theorem \ref{t:colavoid_intro}]
Local existence and uniqueness follow from Picard-Lindel\"{o}f theorem. Proposition \ref{p:flock} excludes the possibility of finite-time blowup, which implies that the only restriction to the lifespan of the solution can originate from collisions. Thus, to finish the proof, it suffices to show collision-avoidance. The cases of $p=2$ and  $2<p\leq 3$ have been addressed in \cite{CCMP-17} and  \cite{M-18}, respectively, and thus we only consider $p>3$.

 \medskip
 Aiming at a contradiction, assume that there exists the first time of collision $t_c>0$, and a set of indices $C$, such that for all $i,j\in C$, we have $x_i(t_c)=x_j(t_c)$ and for all $i\in C$ and $j\notin C$, we have $x_i(t_c)\neq x_j(t_c)$. Consequently,  since all position trajectories $t\mapsto x_i(t)$ are Lipschitz continuous (by Proposition \ref{p:flock}), the following statements hold true:
\begin{align*}
    |x_i(t)-x_j(t)| &>0, \quad \mbox{for all}\quad i\neq j\in\{1,...,N\}, \ t\in[0,t_c),\\
    |x_i(t)-x_j(t)| &\to 0\quad \mbox{as}\quad t\nearrow t_c,\quad \mbox{for all}\quad i,j\in C,\\
    |x_i(t)-x_j(t)| &\geq \delta \quad \mbox{for some}\quad \delta>0 \quad \mbox{and all}\quad i\in C,\ j\notin C, \ t\in[0,t_c].
\end{align*}

\noindent
Let us introduce the relative position and velocity within the cluster $C$ as
$$ \|x\|_C(t):=\sqrt{\sum_{i,j\in C}|x_i(t)-x_j(t)|^2},\qquad \|v\|_C(t):=\sqrt{\sum_{i,j\in C}|v_i(t)-v_j(t)|^2}. $$
According to the above definitions we have
\begin{equation}\label{contradict}
    \|x\|_C(t)\xrightarrow{t\nearrow t_c} 0,
\end{equation}
which is the statement that we aim to contradict.

\smallskip
To this end, we begin by stating the system of dissipative differential inequalities (SDDI) satisfied by the relative position and velocity. Its derivation can be found in \cite{M-18}, inequality (10), and follows the same idea as in the linear case of $p=2$, cf. \cite{CCMP-17}. The only difference with \cite{M-18} is that we use the elementary inequality $$ (|a|^{p-2}a-|b|^{p-2}b)\cdot(b-a)\leq 0, $$ valid for $p\geq 3$ (actually $p\geq 2$ suffices) to ensure that the term $J_{22}$ on page 5252 of \cite{M-18} is non-positive, which allows us to completely remove the last term on the right-hand side of inequality (10) in \cite{M-18}. Since the calculation is rather basic, we omit it.
The SDDI reads
\begin{equation}\label{e:sddi}
\begin{split}
    \left| \frac{\dd}{\dd t}\|x\|_C(t) \right| &\leq \|v\|_C(t),\\
    \frac{\dd}{\dd t}\|v\|_C(t) &\leq -c_0\|x\|_C^{-\alpha}(t)\|v\|_C^{p-1}(t)+c_1\|x\|_C(t),\quad \mbox{for some positive constants }\ c_0, c_1. 
\end{split}
\end{equation}

\noindent
In the remainder of the proof, we assume additionally that 
\begin{equation}\label{e:xleq1}
    \xc\leq 1,
\end{equation}
since whenever $\xc(t_n)>1$, for some $t_n\to t_c$, there is no possibility of collision, due to the fact that the speed is bounded, cf. Proposition \ref{p:flock}.
On the time interval $[0,t_c)$, the functions $t\mapsto \xc^2(t)$ and $t\mapsto \vc^2(t)$ are smooth and $\xc$ is strictly positive. Thus the derivative
$$ \frac{\dd}{\dd t}-\ln(\xc(t)) = -\frac{\frac{\dd}{\dd t}\xc(t)}{\xc(t)}\leq \frac{\vc(t)}{\xc(t)} $$
exists for all $t\in[0,t_c)$. Since we assume that $p\geq 3$, there exists $\eta_0>0$ such that
 \begin{equation}\label{e:etazero}
     c_0\theta^{p-3} - \frac{c_1}{\theta^2} \geq 1,\quad \mbox{for all }\ \theta\geq \eta_0.
 \end{equation}
Fix such $\eta_0$. We shall prove that for $t\in[0,t_c)$, we have
\begin{equation}\label{e:tu}
    \frac{\vc(t)}{\xc(t)} \leq \eta:=\max\left\{\eta_0,\frac{\vc(0)}{\xc(0)}\right\}.
\end{equation}

\noindent
Then clearly

$$\limsup_{t\to t_c}-\ln(\xc(t))\leq -\ln(\xc(0)) + \eta t_c<+\infty$$
ensuring that \eqref{contradict} cannot hold true.

\smallskip
\noindent
 Suppose to the contrary of \eqref{e:tu}, that there exists $t_1\in[0,t_c)$, such that  $\frac{\vc(t_1)}{\xc(t_1)} > \eta$. Denote
$$t_0:= \sup\left\{t<t_1:\ \frac{\vc(t)}{\xc(t)} \leq \eta\right\}.$$
Since $t\mapsto \frac{\vc(t)}{\xc(t)}$ is continuous on $[0,t_c)$ and $\frac{\vc(0)}{\xc(0)} \leq \eta$, it is clear that $0\leq t_0<t_1$ and $\frac{\vc(t_0)}{\xc(t_0)} \leq \eta$, and finally that $\frac{\vc(t)}{\xc(t)}>\eta$ for all $t\in(t_0,t_1)$.

\smallskip
\noindent
Using the SDDI \eqref{e:sddi} on the time interval $(t_0,t_1)$ we infer that
\begin{align*}
    \frac{\dd}{\dd t}\vc&\leq -c_0\xc^{-\alpha}\vc^{p-1} + c_1\xc \\
    & = -c_0\left(\frac{\vc}{\xc}\right)^{p-3}\ \vc^2\ \xc^{p-\alpha-3} + c_1\xc\\
    &\leq -c_0\eta^{p-3}\frac{\vc^2}{\xc} + c_1\xc,
\end{align*}
where the last inequality follows from the fact that, since $\xc\leq 1$ by \eqref{e:xleq1}, the expression on the right-hand side above is maximized whenever singularity in $\xc^{p-\alpha-3}$ is the weakest. By the theorem assumptions, it is precisely the case $p=\alpha+2$.
Since $\vc>\eta\xc>0$ on $(t_0,t_1)$ we can divide both sides of the above inequality by $\vc$ obtaining
$$ \frac{\frac{\dd}{\dd t}\vc}{\vc}\leq -c_0\eta^{p-3}\frac{\vc}{\xc} + c_1\frac{\xc}{\vc}.$$
Observe further that on $(t_0,t_1)$, we have
 $\frac{\xc}{\vc}\leq \frac{1}{\eta}< \frac{1}{\eta^2}\frac{\vc}{\xc},$
which, together with the first inequality in the SDDI \eqref{e:sddi} and with \eqref{e:etazero}, implies that
$$ \frac{\dd}{\dd t}\ln(\vc)\leq -\left(c_0\eta^{p-3} - \frac{c_1}{\eta^2}\right)\frac{\vc}{\xc}\leq -\frac{\vc}{\xc}\leq \frac{d}{\dd t}\ln(\xc).$$

\noindent
Integrating over $(t_0,t_1)$ we obtain
$$ \ln(\vc(t_1))\leq \ln(\vc(t_0)) + \big(\ln(\xc(t_1))-\ln(\xc(t_0))\big)$$
or equivalently
$$ \frac{\vc(t_1)}{\xc(t_1)}\leq \frac{\vc(t_0)}{\xc(t_0)}\leq \eta. $$
This is a contradiction, since $t_1$ was supposed to satisfy $\frac{\vc(t_1)}{\xc(t_1)}>\eta$. Altogether \eqref{e:tu} holds true for all $t\in[0,t_c)$ ensuring that there is no possibility of finite-time collision.
\end{proof}

\medskip
\noindent

%\jp{ I decided to change it into a proposition, since it became maybe a little too long for a remark.}
%{\color{blue}
\begin{prop}\rm\label{rem:collision}
Theorem \ref{t:colavoid_intro} is optimal in the sense that for a fixed $\alpha \geq 1$, if $p>\alpha +2$, then there exist solutions admitting collisions in finite time.
\end{prop}
%}
\begin{proof}
Consider a two-particle variant of system \eqref{e:micro} on the real line. Defining $r:=x_2-x_1$, the system reduces to the ODE:
\begin{equation}\label{e:twopart}
    \ddot{r} = -\dot{r}|\dot{r}|^{p-2} |r|^{-\alpha}.
\end{equation}
We choose initial data such that $r(0)>0$ and $\dot{r}(0)<0$. Assuming no collision occurs, we have $r(t)>0$ and $\dot{r}(t)<0$ for all $t\geq 0$. Then \eqref{e:twopart} can be rewritten as
$$ \ddot{r} = -\dot{r}(-\dot{r})^{p-2} r^{-\alpha} $$
and finally as
\begin{equation}\label{e:collision-eq}
    \ddot{r}(-\dot{r})^{2-p} = -\dot{r}r^{-\alpha}.
\end{equation}
Now we consider two cases $\alpha > 1$ and $\alpha = 1$. Let us first deal with the case $\alpha > 1$. Integrating both sides of the above equation, we obtain
$$ \frac{1}{p-3}(-\dot{r})^{3-p} = \frac{1}{\alpha-1}r^{1-\alpha} + c, $$
where $c=\frac{1}{p-3}(-\dot{r}(0))^{3-p} -  \frac{1}{\alpha-1}r^{1-\alpha}(0)$ and we pick the initial data so that $c=0$. This leaves us with
$$ \dot{r} = -\left(\frac{p-3}{\alpha-1}\right)^\frac{1}{3-p} r^\frac{\alpha-1}{p-3}.$$
Solving the above ODE, we get
$$ r(t) = \left( r^{p-2-\alpha}(0) - \frac{p-3}{\alpha-1}t \right)^\frac{1}{p-2-\alpha}, $$
which for $p>\alpha+2>3$ clearly leads to a collision. 

Next, consider the case $\alpha = 1$. Similarly, integrating \eqref{e:collision-eq}, we derive 
\[
\frac{1}{p-3}(-\dot{r})^{3-p} = - \ln r
\]
by choosing the initial data so that $\frac{1}{p-3}(-\dot{r}(0))^{3-p} + \ln r(0) = 0$. Note that under this initial assumption, we should impose $r(0) < 1$ and thus $r(t) < 1$ for all $ t\geq 0$. This yields
\[
\left((3-p)\ln r \right)^{\frac1{p-3}}\,\dd r = -\dd t,
\]
and thus
\[
\int_{r(0)}^{r(t)} \left((3-p)\ln s \right)^{\frac1{p-3}}\,\dd s = -t.
\]
On the other hand, we find
\begin{align*}
\int_{r(0)}^{r(t)} \left((3-p)\ln s \right)^{\frac1{p-3}}\,\dd s &= r(t)\left((3-p)\ln r(t) \right)^{\frac1{p-3}} - r(0)\left((3-p)\ln r(0) \right)^{\frac1{p-3}} \\
&\quad + \int_{r(0)}^{r(t)} \left((3-p)\ln s\right)^{\frac1{p-3} - 1}\,\dd s \\
&\geq r(t)\left((3-p)\ln r(t) \right)^{\frac1{p-3}} - r(0)\left((3-p)\ln r(0) \right)^{\frac1{p-3}} \\
&\quad + \frac1{(3-p)\ln r(0)}\int_{r(0)}^{r(t)} \left((3-p)\ln s \right)^{\frac1{p-3}}\,\dd s,
\end{align*}
where we used $(3-p)\ln r(0) \leq (3-p)\ln s \leq (3-p)\ln r(t) $ for $r(t) \leq s \leq r(0) < 1$ and $p>3$. This implies
\begin{align*}
r(t)\left((3-p)\ln r(t) \right)^{\frac1{p-3}} &\leq r(0)\left((3-p)\ln r(0) \right)^{\frac1{p-3}} + \left(1 - \frac1{(3-p)\ln r(0)}  \right)\int_{r(0)}^{r(t)} \left((3-p)\ln s \right)^{\frac1{p-3}}\,\dd s \\
&= r(0)\left((3-p)\ln r(0) \right)^{\frac1{p-3}} - \left(1 - \frac1{(3-p)\ln r(0)}  \right)t.
\end{align*}
Thus if we choose $r(0)>0$ small enough such that $(3-p)\ln r(0) > 1$, i.e. $r(0) < e^{-1/(p-3)}$, then $r(t)$ must reach zero in finite time since the left-hand side of the above inequality is nonnegative.
\end{proof}

\section{Compactness of atomic solutions}\label{sec:compact}
We dedicate this section to compactness results for the sequence of atomic solutions $\{\mu^N\}_{N=1}^\infty$. Ultimately, our goal is to extract a weakly convergent subsequence and prove that the local quantities of its limit solve \eqref{e:macro} in the sense of Definition \ref{d:weakEA}. The first compactness result is independent of the nonlinear character of velocity couplings and follows from its $p=2$ variant in \cite{FP-23}. We omit the proof.

\begin{prop}\label{p:compactness}
Fix $T>0$ and let $\{\mu^N\}_{N=1}^\infty$ be any sequence of atomic solutions (cf. Definition \ref{d:empirical}), uniformly compactly supported in $(T+1)B(M)\times B(M)$ and $\rho^N$ -- their local density defined in \eqref{loc_quant}.
Then the following assertions hold.
\begin{enumerate}[label=(\roman*)]
    \item The sequence $\{\mu^N\}_{N=1}^\infty$ is narrowly relatively compact in ${\mathcal M}_+([0,T]\times\R^{2d})$, and any of its limits has the 1D Lebesgue measure as its $t$-marginal. In other words
    \begin{equation}\label{e:odnosze}
    \mu^N\rightharpoonup \mu_t\otimes\lambda^1(t)
    \end{equation}
    up to a subsequence.
   % \item For each $t\in[0,T]$, $\mu_t^N$ is weakly-$\ast$ compact (in the sense of weak-$\ast$ compactness of measures).
    \item The sequence $\{\rho^N\}_{N=1}^\infty$ treated as measure-valued functions on $[0,T]$ is relatively compact in $${C}\left( [0,T]; (\mathcal{P}(\R^d), d_{BL})\right).$$
\end{enumerate}
\end{prop}

While the convergences obtained from compactness ensured by the above proposition work well when passing to the limit with linear terms in \eqref{e:meso}, they are insufficient to deal with the nonlinear term involving the interaction force $F(\mu)$. For this, we need another compactness result, which reads as follows.

\begin{prop}\label{lem:mutimesmu}
Under the assumptions of Proposition \ref{p:compactness} let $\mu^{N_k}$ be a subsequence convergent to $\mu\in {\mathcal M}_+([0,T]\times\R^{2d})$.
Then, up to a subsequence, $[\mu^{N_k}_t\otimes\mu^{N_k}_t]\otimes\lambda^1(t)$ converges narrowly to some $\lambda\in{\mathcal M}_+([0,T]\times\R^{4d})$, which can be disintegrated as 
$$
[\mu_t^{N_k} (x,v) \otimes \mu_t^{N_k} (x',v')] \otimes \lambda^1(t) \xrightharpoonup{k \rightarrow \infty} \lambda(t,x,v,x',v') =  [\mu_t(x,v) \otimes \mu_t(x', v')] \otimes \lambda^1(t).
$$
\end{prop}

\begin{proof}
Since $\{[\mu_t^{N_k} \otimes \mu_t^{N_k}] \otimes \lambda^1\}_{k\in{\mathbb N}}$ is a sequence of uniformly compactly supported probability measures, its narrow compactness follows directly from Banach-Alaoglu theorem ensuring the existence of a subsequence (not relabeled) and a measure $\lambda \in \mathcal{M}_+([0,T]\times \R^{4d})$ such that 
\begin{equation}\label{e:mutimesmu.2}
[\mu_t^{N_k}\otimes \mu_t^{N_k}] \otimes \lambda^1(t) \xrightharpoonup{k \rightarrow \infty} \lambda.
\end{equation}
To prove that $\lambda$ disintegrates as 
$$
\lambda(t,x,v,x',v') = \lambda_t (x,v,x',v') \otimes \lambda^1(t),
$$
suppose that $\nu$ is the $t$-marginal of $\lambda$ (so that $\lambda=\lambda_t\otimes\nu(t)$) and test the narrow convergence with $g(t,x,v,x',v') = \phi(\pi_t(t,x,v,x',v')) = \phi(t)$, with an arbitrary continuous function $\phi$, obtaining
\begin{align*}
\int_0^T  \phi(t) \dd t &= \int_0^T \int_{\R^{4d}} g \dd [\mu_t^N\otimes\mu_t^N] \dd t \xrightarrow{N\rightarrow \infty} \int_0^T \int_{\R^{4d}} g \dd \lambda\\
&= \int_0^T  \phi(t) \int_{\R^{4d}}\dd \lambda_t \otimes \nu(t) = \int_0^T \phi(t) \dd\nu(t).
\end{align*}
Since $\phi$ is arbitrary, we conclude that $\nu(t) = \lambda^1(t)$.
It remains to prove that 
\begin{equation}\label{nielin}
    \lambda_t(x,v,x',v') = \mu_t(x,v) \otimes \mu_t(x',v')\quad \mbox{for a.a. } t\in[0,T].
\end{equation}

\medskip

\noindent Following \eqref{e:boundM}, let $\Omega \subset\subset (T+1)B(M)\times B(M)$ be a superset of all supports of $\mu^{N_k}_t$ and $\mu_t$ for all $t\in[0,T]$ (observe that $\{\mu_t\}_{t\in[0,T]}$ is compactly supported as a narrow limit of $\{\mu^{N_k}_t\}_{t\in[0,T]}$).   Furthermore, let ${\mathcal D}$ be a countable dense (in the sense of uniform convergence) subset of $C^\infty_c(\Omega)$ and fix any 
$$
\phi \in {\mathcal D} \subset C^\infty_c(\Omega),\qquad  \|\phi\|_{C^2}\leq 1. 
$$

\noindent
For each $k$, consider the function
$$ t\mapsto \int_{\R^{2d}}\phi\dd\mu_t^{N_k} = \frac{1}{N}\sum_{i=1}^N\phi(x_i^{N_k}(t),v_i^{N_k}(t)), $$
where $(x_i^{N_k},v_i^{N_k})_{i=1}^{N_k}$ is the solution of the particle system \eqref{e:micro} corresponding to $\mu^{N_k}$. The above function is smooth since the solutions to the particle system are smooth. Thus, applying the standard trick with the exchange of indices $i$ with $j$, we obtain
\begin{equation}\label{e:mutimesmu.1}
\begin{split}
    \left| \frac{d}{dt}\int_{\R^{2d}}\phi\dd\mu_t^{N_k} \right| 
   &\le \frac{1}{N_k}\sum\limits_{i=1}^{N_k} \left| \nabla_x \phi(x_i^{N_k},v_i^{N_k}) \right| \left| v_i^{N_k}\right|\\
   &\quad+ 
   \frac{1}{2N_k^2}\sum\limits_{i\neq j=1}^{N_k}\underbrace{\frac{\left|\nabla_v\phi(x_i^{N_k}, v_i^{N_k}) - \nabla_v\phi(x_j^{N_k}, v_j^{N_k})\right| \left|v_j^{N_k}-v_i^{N_k}\right|^{p-1}}{|x_i^{N_k} - x_j^{N_k}|^\alpha}}_{S:=}.
\end{split}
\end{equation}

\noindent
Using Lipschitz continuity of $\nabla_v\phi$,
by Young's inequality with exponents $\frac{1}{p}+\frac{1}{p'}=1$, the singular term $S$ on the right-hand side above can be upper-bounded by
\begin{align*}
   S&\leq [\nabla_v\phi]_{Lip}\left(|x_i^{N_k}-x_j^{N_k}|^{1-\alpha}|v_i^{N_k}-v_j^{N_k}|^{p-1}+\frac{|v_j^{N_k}-v_i^{N_k}|^p}{|x_i^{N_k} - x_j^{N_k}|^\alpha}\right)\\
   &=  [\nabla_v\phi]_{Lip}\left( |x_i^{N_k}-x_j^{N_k}|^{1-\frac{\alpha}{p}}\frac{|v_i^{N_k}-v_j^{N_k}|^{p-1}}{|x_i^{N_k}-x_j^{N_k}|^\frac{\alpha}{p'}}+\frac{|v_j^{N_k}-v_i^{N_k}|^p}{|x_i^{N_k} - x_j^{N_k}|^\alpha} \right)\\
   &\leq [\nabla_v\phi]_{Lip}\left( \frac{1}{p}|x_i^{N_k}-x_j^{N_k}|^{p-\alpha}+\left(1+\frac{1}{p'}\right)\frac{|v_j^{N_k}-v_i^{N_k}|^p}{|x_i^{N_k} - x_j^{N_k}|^\alpha} \right).
\end{align*}
Plugging the above inequality into \eqref{e:mutimesmu.1} and using the bounds of the diameter of $\Omega$ as well as the fact that $\|\phi\|_{C^2}\leq 1$, yields
$$  \left| \frac{d}{\dd t}\int_{\R^{2d}}\phi\dd\mu_t^{N_k} \right| \leq M+ \frac{1}{2p}\frac{1}{N_k^2}\sum_{i\neq j}^{N_k}|x_i^{N_k}-x_j^{N_k}|^{p-\alpha} + \frac{1}{2}\left(1+\frac{1}{p'}\right)D_p[\mu_t^{N_k}]. $$
Integrating the above over $[s,t]$ we obtain
\begin{equation}\label{e:mutimesmu.3}
\begin{split}
\left|\int_{\R^{2d}}\phi\dd\mu_t^{N_k} - \int_{\R^{2d}}\phi\dd\mu_s^{N_k}\right|
     &\stackrel{\star}{\leq} MT+(2T(T+1)M)^{p-\alpha} +\frac{1}{2}\left(1+\frac{1}{p'}\right) E[\mu_0^{N_k}] \\
     &\leq MT+(2T(T+1)M)^{p-\alpha} + \frac{1}{2}\left(1+\frac{1}{p'}\right)(2M)^p.
\end{split}
\end{equation}
Here inequality $\stackrel{\star}{\leq}$ follows from the fact that $p\geq \alpha$ and from the energy equality in Proposition \ref{p:eneq}  ($E$ and $D_p$ are the kinetic energy and its dissipation rate defined in \eqref{kinen}).
Inequality \eqref{e:mutimesmu.3} ensures that the sequence of functions 
$$t\mapsto \int_{\R^{2d}}\phi\dd\mu_t^{N_k}$$
is bounded in $W^{1,1}([0,T])$, and thus, it is compact in $L^2([0,T])$. Therefore, for fixed $\varphi$, we have (up to a subsequence)
$$
\int_{\R^{2d}}\phi \dd\mu^{N_{k}}_t\longrightarrow \psi_\phi(t)\quad \text{in}\quad L^2([0,T]),
$$
where $\psi_\phi$ is some $\phi$-dependent limit. Proceeding diagonally, there exists a  subsequence of $N_{k}$ (which we shall not relabel), such that
$$
\forall \phi \in {\mathcal D},\ \|\phi\|_{C^2}\leq 1 \quad\mbox{we have}\quad \int_{\R^{2d}}\phi \dd\mu^{N_{k}}_t \longrightarrow \psi_\phi(t) \text{ in }L^2([0,T]).
$$

\noindent
By the assumptions, sequence $\mu^{N_k}$ converges narrowly as in \eqref{e:odnosze}, and thus for all $\psi\in C([0,T])$ we have
$$
\int_0^T \psi(t)\int_{\R^{2d}} \phi\dd\mu^{N_{k}}_t \dd t = \int_0^T\int_{\R^{2d}} \psi(t) \phi(x,v) \dd \mu_t^{N_{k}} \dd t \xrightarrow{N_{k} \rightarrow \infty}
\int_0^T \psi(t) \int_{\R^{2d}} \phi \dd \mu_t \dd t.
$$
Since $\psi$ belongs to a dense subspace $C([0,T])$ of $L^2([0,T])$, by uniqueness of the weak limit, we deduce that
$$
\psi_\phi(t) = \int_{\R^{2d}}\phi\dd\mu_t\quad \mbox{as elements of }\ L^2([0,T])\ \mbox{ for all } \phi\in{\mathcal D}\ \mbox{with}\ \|\phi\|_{C^2}\leq 1.
$$
Now suppose that $\phi_1$ and $\phi_2$ are two such functions. Then
$$ \left(t\mapsto \int_{\R^d}\phi_1\dd\mu^{N_{k}}_t\, \int_{\R^d}\phi_2\dd\mu^{N_{k}}_t\right) \xrightarrow{k\to\infty}\left(t\mapsto\int_{\R^{2d}}\phi_1\dd\mu_t\, \int_{\R^{2d}}\phi_2\dd\mu_t\right) $$
strongly in $L^1([0,T])$. Since $C([0,T]) \subset L^\infty([0,T]) = L^1([0,T])^*$, we infer (by weak $L^1$ convergence) that for all $\psi\in C([0,T])$,
\medskip
\begin{align*}
\int_0^T \int_{\R^{4d}} \psi(t) \phi_1(x,v) \phi_2(x',v') \dd \left[\mu_t^{N_{k}}(x,v)\otimes \mu_t^{N_{k}}(x',v')\right] \dd t =
\int_0^T \psi(t) \int_{\R^{2d}}\phi_1\dd\mu^{N_{k}}_t \int_{\R^{2d}}\phi_2\dd \mu^{N_{k}}_t\dd t \\
\rightarrow \int_0^T \psi(t) \int_{\R^{2d}}\phi_1\dd\mu_t \int_{\R^{2d}}\phi_2\dd \mu_t\dd t = 
\int_0^T \int_{\R^{4d}} \psi(t) \phi_1(x,v) \phi_2(x',v') \dd\left[\mu_t(x,v)\otimes \mu_t(x',v')\right]\dd t.
\end{align*}
By a standard density argument based on Stone-Weierstrass theorem the above convergence holds true for all $g\in C([0,T]\times\Omega^2)$ implying that 
\begin{equation*}
[\mu_t^{N_k}\otimes \mu_t^{N_k}] \otimes \lambda^1(t) \xrightharpoonup{k \rightarrow \infty}[\mu_t\otimes \mu_t]\otimes\lambda^1(t) .
\end{equation*}

\noindent Finally, by the uniqueness of the narrow limit, and by \eqref{e:mutimesmu.2} we conclude that
$$\lambda = [\mu_t\otimes \mu_t]\otimes\lambda^1(t)$$
and the proof is finished.
\end{proof}

\section{Micro- to macroscopic mean-field limit}\label{sec:mflimit}

The goal of this section is to prove the main result - Theorem \ref{t:EAexist}. The first issue is to prepare initial data for the mean-field limit. Since this procedure is still the same as in the classical setting, we only provide an extract here, referring to \cite{FP-23} for further details.

Following the notation of Theorem \ref{t:EAexist}, let $\rho_0 \in \mathcal{P}(\R^d_x)$ be fixed, with bounded support compactly contained in some ball $B(M)$ (for $M>0$ sufficiently large) and assume $u_0:\R^d_x \rightarrow \R^d$ satisfies $u_0 \in L^\infty(\rho_0)$, $\lVert u_0 \rVert_{L^\infty(\rho_0)} \le M$. We approximate $\rho_0$ and $u_0$ with a sequence of empirical measures of the form
\begin{equation}\label{approxini}
    \begin{aligned}
    \rho_0^N(x) &:= \frac{1}{N} \sum\limits_{i=1}^N\delta_{x_{i0}^N}(x) \xrightharpoonup{N \rightarrow \infty} \rho_0, & x_{i0}^N\neq x_{j0}^N\in\R^d \text{ for }i\neq j,\\
m_0^N(x) &:= \frac{1}{N} \sum_{i=1}^N v_{i0}^N\delta_{x_{i0}^N}(x) \xrightharpoonup{N \rightarrow \infty} u_0\rho_0, & v_{i0}^N\in\R^d,
\end{aligned}
\end{equation}
where $m_0^N$ and $u_0 \rho_0$ are understood as vector measures in $(\mathcal{M}(\R^d))^d$. We point out the fact that since $u_0$ is only defined $\rho_0$-a.e., it is not immediately clear how to choose points $x_{i0}^N$; this issue is addressed in \cite{FP-23} where the argument is based on Tietze theorem. The points $\left( x_{i0}^N, v_{i0}^N\right)_{i=1}^N$ serve as good, non-collisional initial data for the particle system \eqref{e:micro} and in a natural way define the corresponding unique, classical, non-collisional solution $\left(x_i^N(t), v_i^N(t)\right)_{i=1}^N$ granted by Theorem \ref{t:colavoid_intro}. Following Definition \ref{d:empirical}. we obtain the atomic solution $\mu^N \in \mathcal{M}([0,T]\times \R^{2d})$. Properties of the particle system guarantee that both compactness results of Section 3 apply to $\{ \mu^N \}_{N=1}^\infty$. Moreover, we have the following lemma.

\begin{lem}\label{l:energy_convergence}
    Suppose $\mu^N$ is a sequence of atomic solutions satisfying Propositions \ref{p:compactness} and \ref{lem:mutimesmu} and let $\mu$ be its narrow limit. Moreover, assume $\mu^N$ are uniformly (with respect to $N$) compactly supported, i.e. there exists $M>0$ such that for all $N$ and $t\in[0,T]$
    $$
    \text{spt}(\mu_t^N) \subset \subset (T+1)B(M) \times B(M).
    $$
    Then $\mu$ is compactly supported in $(T+1)B(M)\times B(M)$ and there exists a subsequence $\{\mu_t^{N_k}\}_{k=1}^\infty$ such that 
    $$
    E[\mu_t^{N_k}] \xrightarrow{k \rightarrow \infty} E[\mu_t] \text{ for all } t \in A_E,
    $$
    where $E$ is the kinetic energy introduced in \eqref{kinen}, and $A_E$ is a full measure subset of $[0,T]$.
\end{lem}

\begin{proof}
    The fact that $\mu_t$ is compactly supported in the same superset as each $\mu_t^N$ follows from the definition of narrow convergence and testing with a function supported outside the set $(T+1)B(M)\times B(M)$. As for the convergence of the energy, note that in the first part of Proposition \ref{lem:mutimesmu} we have established that 
    $$
    \int_{\R^{2d}} \phi \dd \mu^N_t \rightarrow \int_{\R^{2d}} \phi \dd \mu_t \quad \text{in} \quad L^2([0,T]), \text{ hence also for a.a. } t \in [0,T]
    $$
    (up to a subsequence) for all $\phi=\phi(x,v) \in C^\infty_c$. Due to the uniform compactness of all involved measures, the function $\phi = |v|^2$, appearing in the definition of $E$, can be treated as a $C^\infty_c$ function, which completes the proof.
\end{proof}

Following the above lemma and results from the previous section, we collect all the necessary properties of approximative solutions in the following definition.

\begin{definition}[Approximative solution]\label{d:app_sol}
    Fix $T>0$. For a given bounded, compactly supported $\rho_0 \in \mathcal{P}(\R^d)$ and $u_0 \in L^\infty(\rho_0)$, $u_0: \R^d_x \rightarrow \R^d$, we define initial approximation of the form \eqref{approxini}. For each $N\in \mathbb{N}$, let $(x_i(t), v_i(t))^N_{i=1}$ be the solution of the particle system \eqref{e:micro} and $\mu^N$ be a corresponding atomic solution. Based on Lemma \ref{l:energy_convergence} and results of Section 3, we extract (without relabelling) a subsequence of $\{\mu^N\}_{N=1}^\infty$ satisfying the following assertions:
    \begin{align*}
     \mu^N(t,x,v)&\xrightharpoonup{N \rightarrow \infty} \mu(t,x,v)= \mu_t(x,v) \otimes \lambda^1(t); \\
     \left[ \mu_t^N(x,v) \otimes \mu_t^N(x',v')\right] \otimes \lambda^1(t) &\xrightharpoonup{N \rightarrow \infty} \left[ \mu_t(x,v)\otimes \mu_t(x',v')\right]\otimes \lambda^1(t); \\
     \int_{\R^d_v} \dd \mu_t^N(x,v) =: \rho^N_t(x) &\xrightarrow{N \rightarrow \infty} \rho_t(x) \text{ in }C\left([0,T]; ({\mathcal P}(\R^d), d_{BL})\right);\\
     E[\mu_t^N] &\xrightarrow{N\rightarrow\infty} E[\mu_t] \text{ on } A_E,\quad \mbox{(see kinetic energy defined in \eqref{kinen}),}
\end{align*}
with $A_E$ defined in Lemma \ref{l:energy_convergence}. Moreover, $\mu_t$ is compactly supported for each $t\in [0,T]$, with bounds on support inherited from $\mu_t^N$.
\end{definition}

\begin{rem}\rm\label{r:rhoagrees}
    The measure $\mu_t$ disintegrates as in \eqref{mudis} for a.a. $t\in[0,T]$, that is 
    $$
    \mu_t(x,v) = \sigma_{t,x}(v) \otimes \rho_t(x)\quad \text{for a.a. }t\in[0,T].
    $$
    More importantly, the $x-$marginal of $\mu_t$ defined as the local density $\rho_t$ coincides for a.a. $t\in[0,T]$ with the limit of $\rho^N_t$. This follows from the same procedure as in the proof of Proposition \ref{lem:mutimesmu}. 
\end{rem}

\subsection{Monokineticity of the limit family $\{\mu_t\}_{t \in [0,T]}$} Next goal in this section is to establish that the family $\{\mu_t \}_{t \in [0,T]}$ defined above fits into the framework of Theorem \ref{main1} and is therefore monokinetic. We need to check whether $\{\mu_t\}_{t \in [0,T]}$ satisfies Definition \ref{d:MP+SF} and condition \eqref{main1-as1}.

We begin with \eqref{main1-as1}. Observe that since $\mu_t$ is compactly supported in $(T+1)B(M)\times B(M)$ and $p \le \alpha+2$, we have
\begin{align*}
    \int_0^T \int_{\R^{4d}\setminus \Delta} \frac{|v-v'|^{\alpha+2}}{|x-x'|^\alpha} \dd [ \mu_t \otimes \mu_t ] \dd t &\le
    \int_0^T \int_{\R^{4d}\setminus \Delta} |v-v'|^{\alpha+2-p} \frac{|v-v'|^p}{|x-x'|^\alpha} \dd [ \mu_t \otimes \mu_t] \dd t \\
    &\le (2M)^{\alpha+2 - p} \int_0^T D_p[\mu_t] \dd t
\end{align*}
meaning we can estimate $D^\alpha$ from Theorem \ref{main1} with the energy dissipation rate $D_p$ (cf. \eqref{kinen}). We establish the control over the latter with the following lemma.

\begin{lem}\label{l:mu_dissipate}
    The family $\left\{ \mu_t \right\}_{t \in [0,T]}$ defined above dissipates kinetic energy almost everywhere, i.e. the function $t \mapsto E[\mu_t]$ is non-increasing for a.a. $t_1, t_2 \in [0,T]$ and 
    \begin{equation}
        \int_{t_1}^{t_2} D_p[\mu_t] \dd t \le E[\mu_{t_1}] - E[\mu_{t_2}].
    \end{equation}
\end{lem}

\begin{proof}
    Observe that by Proposition \ref{p:eneq} each $\mu^N$ satisfies for all $0\le t_1 \le t_2 \le T$
    $$
    E[\mu_{t_1}^N] - E[\mu_{t_2}^N] = \int_{t_1}^{t_2} D^p[\mu_t^N] \dd t \ge 0.
    $$
    We restrict our attention to $t_1, t_2 \in A_E$ from Definition \ref{d:app_sol}, which allows us to pass with $N \rightarrow \infty$ on the left hand side, obtaining
    $$
    E[\mu_{t_1}] - E[\mu_{t_2}] = \lim\limits_{N \rightarrow \infty} \int_{t_1}^{t_2} D_p[\mu_t^N] \dd t  \ge \liminf\limits_{N\rightarrow \infty} \int_{t_1}^{t_2} D_p[\mu_t^N] \dd t.
    $$
    Since the singular integrand defining $D_p$ is nonnegative and lower-semicontinuous, the Portmanteau theorem together with the narrow convergence  $\mu_t^N \rightharpoonup \mu_t$, allows us to further estimate the right-hand side:
    $$
    \liminf\limits_{N \rightarrow \infty}\int_{t_1}^{t_2} D_p [\mu_t^N] \dd t \ge \int_{t_1}^{t_2} D_p [\mu_t] \dd t \ge 0.
    $$
    Since $A_E \subset [0,T]$ is of full measure, this concludes the proof.
\end{proof}

\noindent
Next (and last), we need to ensure that $\{\mu_t\}_{t \in [0,T]}$ satisfies Definition \ref{d:MP+SF}.

\begin{prop}\label{p:syf}
     For $\alpha \le p$, the family of measure $\{\mu_t\}_{t \in [0,T]}$ from Definition \ref{d:app_sol} satisfies time-regularity conditions (MP) and (SF) from Definition \ref{d:MP+SF}.
\end{prop}

\noindent We relegate the technical proof to the Appendix.

% { \color{blue}\mf{Do poprawy po zmienieniu w Preliminaries } \jp{JP}
% Next we check that $\{\mu_t\}$ is locally mass preserving. Since this condition concerns only $\rho_t$, which is independent of the nonlinear character of velocity couplings, the proof remains the same as in the classical $p=2$ setting in [xxFPxx]; we omit the details.
% \begin{prop}\label{p:mu_is_LMP}
%     The family of measures $\{\mu_t\}_{t\in[0,T]}$ from Definition \ref{d:app_sol} is locally mass preserving.
% \end{prop}

% Finally, we proceed with the last condition of Theorem \ref{main1}.
% \begin{prop}\label{p:syf}
%     is steadily flowing.
% \end{prop}

% The idea of the proof follows the same steps as in [xxFPxx]. The only difference is that we use H\" older inequality with $\frac{1}{p}+\frac{1}{p'}=1$ instead of Cauchy-Schwarz inequality. For the reader's convenience we include the proof in the appedndix.
% }

\subsection{Existence of solutions to the $p$-Euler-alignment system}

Recall that due to our effort in the previous subsection, we know that the family $\{\mu_t \}_{t \in [0,T]}$ satisfies the assumptions of Theorem \ref{main1} and is therefore monokinetic. With this knowledge, we now prove our last main result, Theorem \ref{t:EAexist}.

\begin{proof}
    Fix any $\rho_0 \in \mathcal{P}(\R^d)$ and $u_0 \in L^\infty(\rho_0)$, as in the assumptions of Theorem \ref{t:EAexist}. We approximate it by empirical measures in \eqref{approxini}. The approximation provides initial data for the particle system \eqref{e:micro}. Then, unique solution exists by Theorem \ref{t:colavoid_intro} and consequently a sequence $\{\mu_t^N\}_{t\in[0,T]}$ of atomic solutions exists, which by results of Sections \ref{sec:compact} and \ref{sec:mflimit} converges, up to a subsequence as stated in Definition \ref{d:app_sol}. Thus, we construct the family $\{\mu_t\}_{t \in [0,T]}$. Remembering that the local quantities of $\mu(t,x,v)$ were defined as 
    $$
    \rho_t(x) = \int_{\R^d_v} \dd \mu_t(x,v), \quad u(t,x) = \int_{\R^d_v} v \dd \sigma_{t,x}(v)
    $$
    and following the reasoning in Remark \ref{r:rhoagrees}, we note that $\mu$ admits the disintegration 
    $$
    \mu(t,x,v) = \sigma_{t,x}(v) \otimes \rho_t(x) \otimes \lambda^1(t).
    $$
    By virtue of Definition \ref{d:app_sol}, we know that $\rho_t \in C([0,T], \left( \mathcal{P}(\R^d), d_{BL}\right))$ and, since $\mu$ is uniformly compactly supported, so is $\rho_t$, and thus $u \in L^\infty(\rho)$, where the constants required in item (i) of Definition \ref{d:weakEA} are inherited from the constants related to the support of $\mu_t$.

    To prove the energy inequality in item (ii) of Definition \ref{d:weakEA}, observe that due to the previous reasoning of this section, family $\{\mu_t \}_{t\in[0,T]}$ is monokinetic, so its decomposition simplifies to 
    \begin{equation}\label{e:mu_disint}
        \mu(t,x,v) = \delta_{u(t,x)}(v) \otimes \rho_t(x) \otimes \lambda^1(t). 
    \end{equation}
    Plugging this decomposition into the energy inequality in Lemma \ref{l:mu_dissipate} yields the result.

    It remains to prove that the pair $(\rho, u)$ satisfies, in a weak sense, continuity and momentum equation in Definition \ref{d:weakEA} with our fixed initial data $(\rho_0, u_0)$. Let us test the weak formulation of the kinetic equation \eqref{e:weakkin} for $\mu^N$ with any $\phi = \phi(t,x) \in C^1([0,T] \times \R^{2d})$ to obtain
    \begin{align*}
    0 &= \int_{\R^{2d}} \phi(0,x) \dd \mu_0^N + \int_0^T \int_{\R^{2d}} \left( \partial_t \phi + v \cdot \nabla_x \phi \right) \dd \mu_t^N \dd t \\
    &= \int_{\R^{d}} \phi(0,x) \dd \rho_0^N + \int_0^T \int_{\R^{2d}} \left( \partial_t \phi + v \cdot \nabla_x \phi \right) \dd \mu_t^N \dd t.     
    \end{align*}    
    Definition \ref{d:app_sol} and the approximation of $\rho_0$ of the form \eqref{approxini} allows to pass with $N\rightarrow \infty$. Disintegrating $\mu$ as in \eqref{e:mu_disint} yields the continuity equation.

    As for the momentum equation, we perform the proof componentwise. With the notation $v= (v_1, \ldots, v_d)$, we test the kinetic equation with $v_i \phi(t,x)$, where $\phi$ is as before. We obtain
    \begin{equation}
        \begin{aligned}
            \mathcal{L}^N &:= \int_{\R^{2d}} \phi(0,x) v_i \dd \mu_0^N +\int_0^T \int_{\R^{2d}} \left( v_i \partial_t \phi + v_i \left( v \cdot \nabla_x \right)\phi\right) \dd \mu_t^N \dd t\\
            &= -\frac{1}{2}\int_0^T \int_{\R^{4d}\setminus \Delta} \frac{\left(\phi-\phi'\right) \left(v_i - v_i'\right)|v-v'|^{p-2}}{|x-x'|^\alpha} \dd \left[ \mu_t^N \otimes \mu_t^N\right] \dd t =: \mathcal{R}^N.
        \end{aligned}
    \end{equation}
    For the first term in $\mathcal{L}^N$, we again use \eqref{approxini}:
    $$
    \int_{\R^{2d}} \phi(0,x) v_i \dd \mu_0^N = \frac{1}{N}\sum\limits_{k=1}^N \phi(0,x_{k0}^N)\left( v_{k0}^N\right)_i = \int_{\R^d} \phi(0,x) \dd \left( m_0^N\right)_i \xrightarrow{N \rightarrow \infty} \int_{\R^d} \phi(0,x) u_{i0} \dd \rho_0.
    $$
    Combining this with the convergence established in Definition \ref{d:app_sol} and monokinetic decomposition \eqref{e:mu_disint} yields
    $$
    \mathcal{L}^N \xrightarrow{N \rightarrow \infty} \int_{\R^d} \phi(0,x) u_{i0} \dd \rho_0 + \int_0^T \int_{\R^d} \left( u_i \partial_t \phi + u_i \left( u \cdot \nabla_x \right) \phi \right) \dd \rho_t \dd t =: \mathcal{L}.
    $$
    Next we deal with $\mathcal{R}^N$, which requires more delicate approach due to the singularity of the integrand. For any $m>0$ we have 
    \begin{equation*}
        \begin{aligned}
            \mathcal{R}^N &:= \frac{1}{2}\int_0^T \int_{\R^{4d}\setminus\Delta} \frac{(\phi - \phi')(v_i - v_i')|v-v'|^{p-2}}{\max \{ | x-x'|, m\}^\alpha} \dd \left[ \mu_t^N \otimes \mu_t^N \right] \dd t \\
            &\quad + \frac{1}{2} \bigg( 
            \int_0^T \int_{ (\R^{4d}\setminus \Delta )\cap \{ |x-x'|< m\}} \frac{(\phi-\phi')(v_i-v_i')|v-v'|^{p-2}}{|x-x'|^\alpha} \dd \left[ \mu_t^N \otimes \mu_t^N \right] \dd t \\
            &\quad - \int_0^T \int_{(\R^{4d}\setminus \Delta )\cap \{ |x-x'|< m\} } \frac{(\phi-\phi')(v_i-v_i')|v-v'|^{p-2}}{m^\alpha} \dd \left[ \mu_t^N \otimes \mu_t^N \right] \dd t \bigg) \\
            &=:\mathcal{G}^N(m) + \mathcal{B}^N(m).
        \end{aligned}
    \end{equation*}
    The "good" integrand in $\mathcal{G}^N(m)$ is continuous for any fixed $m>0$ and equal to $0$ on the diagonal set $\Delta$, thus it is defined on whole $\R^{4d}$. Combining this with convergence established in Definition \ref{d:app_sol}, we obtain
    $$
    \mathcal{G}^N(m) \xrightarrow{N\rightarrow \infty} \frac{1}{2}\int_0^T \int_{\R^{4d}\setminus\Delta} \frac{(\phi-\phi')(v_i-v_i')|v-v'|^{p-2}}{\max\{|x-x'|,m\}^\alpha} \dd [\mu_t \otimes \mu_t ]\dd t =: \mathcal{G}(m).
    $$
    This convergence is enough to establish that the "bad" sequence $\mathcal{B}^N(m)$ is convergent as well, with
    $$
    \lim\limits_{N\rightarrow \infty}\mathcal{B}^N(m) = \mathcal{L} - \mathcal{G}(m).
    $$
    We can bound the $\mathcal{B}^N(m)$ further. First, we denote
    $$
    \{ 0 < |x-x'| < m \} := \left(\R^{4d}\setminus\Delta \right) \cap \{ |x-x'|<m \}.
    $$
    Exploiting the fact that $\phi$ is independent of $v$ and Lipschitz continuous in $x$, we use H\"{o}lder inequality similarly as before to upper-bound $\mathcal{B}^N$:
    \begin{equation}\label{e:almost_done}
        \begin{aligned}
            \left\lvert \mathcal{B}^N(m) \right\rvert &\le 
            \int_0^T \int_{\{0<|x-x'|<m\}} \frac{|\phi - \phi'||v_i -v'_i||v-v'|^{p-2}}{|x-x'|^\alpha} \dd [ \mu_t^N \otimes \mu_t^N ] \dd t 
            \\
            &\lesssim \int_0^T \int_{\{0<|x-x'|<m\}} \frac{|v-v'|^{p-1}}{|x-x'|^{\alpha-1}}  \dd [ \mu_t^N \otimes \mu_t^N ] \dd t \\
            &\lesssim \left( \int_0^T \int_{\{0<|x-x'|<m\}} \frac{|v-v'|^p }{|x-x'|^\alpha} \dd [ \mu_t^N \otimes \mu_t^N ] \dd t \right) ^{1/p'} \\
            &\hspace{4cm} \times 
            \left( \int_0^T \int_{\{0<|x-x'|<m\}}  |x-x'|^{p-\alpha} \dd [ \mu_t^N \otimes \mu_t^N ] \dd t \right)^{1/p} \\
            &\le \left(E[\mu_0^N]\right)^{1/p'} \left( \int_0^T \int_{\{0<|x-x'|<m\}}  |x-x'|^{p-\alpha} \dd [ \mu_t^N \otimes \mu_t^N ] \dd t \right)^{1/p},
        \end{aligned}
    \end{equation}
    where we use the energy inequality from Proposition \ref{p:eneq}, while $p,p'$ are H\"{o}lder conjugate. Now the proof differs depending on whether $\alpha < p$ or $\alpha = p$.
    
    \medskip
    \noindent
    $\diamond$ {\it Proof in the case $\alpha < p$.} For $\alpha < p$, from \eqref{e:almost_done} we have
    $$
    \lim\limits_{N\rightarrow \infty} |\mathcal{B}^N(m)| \lesssim \left( T m^{p-\alpha} \right)^{1/p} \xrightarrow{m\rightarrow 0}0,
    $$
    which implies that 
    $$
    \mathcal{L} = \lim\limits_{m\rightarrow 0} \mathcal{G}(m).
    $$
    Returning to the integrals defining $\mathcal{L}$ and $\mathcal{G}$ we have, by the dominated convergence theorem
    \begin{align*}
           &\int_{\R^d} \phi(0,x) u_{i_0} \dd \rho_0 + \int_0^T \int_{\R^d} \left( u_i \partial_t \phi + u_i \left( u\cdot \nabla_x \right) \phi \right) \dd \rho_t \dd t \\
           &\quad =
    -\frac{1}{2}\int_0^T \int_{\R^{4d}\setminus \Delta} \frac{(\phi - \phi')(v_i-v_i')|v-v'|^{p-2}}{|x-x'|^\alpha} \dd [\mu_t \otimes \mu_t ]\dd t.
    \end{align*}
    Since $\mu_t$ is monokinetic, we rewrite the right-hand side as
    \begin{align*}
        & \int_{\R^d} \phi(0,x) u_{i_0} \dd \rho_0 + \int_0^T \int_{\R^d} \left( u_i \partial_t \phi + u_i \left( u\cdot \nabla_x \right) \phi \right) \dd \rho_t \dd t \\
        &\quad =
    -\frac{1}{2}\int_0^T \int_{\R^{2d}\setminus \Delta} \frac{(\phi - \phi')(u_i-u_i')|u-u'|^{p-2}}{|x-x'|^\alpha} \dd [\rho_t \otimes \rho_t ]\dd t.  
    \end{align*}
    Thus the pair $(\rho, u)$ satisfies the momentum equation as well and the proof is finished.

    \medskip
    \noindent
    $\diamond$ {\it Proof in the case $\alpha =p$.} This time around we bound the $\mathcal{B}^N(m)$ in a weaker way, as from \eqref{e:almost_done} we only have 
    \begin{align}\label{e:almost_almost_done}
    \begin{aligned}
         \lim\limits_{N\rightarrow \infty}|B^N(m)|^p &\lesssim \limsup\limits_{N\rightarrow \infty} \int_0^T [\mu_t^N \otimes \mu_t^N]\left( \{0 \le |x-x'| \le m \} \right) \dd t \\
         &\le 
    \int_0^T [\mu_t \otimes \mu_t ]\left( \{ 0 \le |x-x'| \le m \} \right) \dd t.
    \end{aligned}
    \end{align}
    The first inequality is just \eqref{e:almost_done} rewritten for $\alpha = p$ while the latter is justified since $\{0 \le |x-x'| \le m \}$ is closed, thus Portmanteau theorem states that 
    $$
    \limsup_{N \rightarrow \infty} \mu_t^N \otimes \mu_t^N \otimes \lambda^1(t) \left( \{0 \le |x-x'| \le m \} \times [0,T] \right) \le \mu_t\otimes \mu_t \otimes \lambda^1(t) \left( \{0 \le |x-x'| \le m \} \times [0,T] \right)
    $$
    as $\mu_t^N \otimes \mu_t^N \otimes \lambda^1(t) \rightharpoonup \mu_t\otimes \mu_t \otimes \lambda^1(t)$ by Proposition \ref{lem:mutimesmu}. Passing with $m \rightarrow 0$ in \eqref{e:almost_almost_done}, we arrive at 
    $$
    \lim\limits_{m \rightarrow 0} \lim\limits_{N\rightarrow \infty} |\mathcal{B}^N(m)| \le \left( \int_0^T [\mu_t \otimes \mu_t](\Delta) \dd t \right)^{1/p} = 
    \left( \int_0^T \sum\limits_{k=1}^\infty  (\rho_k(t))^2 \right)^{1/p},
    $$
    where $\rho_k(t)$ is the mass of $k$-th atom of $\rho$ at time $t$. In part (B) of Theorem \ref{t:EAexist}, we assumed that $\rho$ is non-atomic for $a.a.$ $t \in [0,T]$; in such a case the right-hand side above is $0$ and the rest of the proof follows the same path as in the case $\alpha<p$. Altogether, the proof is finished.
\end{proof}

\appendix
%\section{Disintegration theorem}\label{sec:App2}

\section{Proof of Proposition \ref{p:syf}}
\begin{proof}[Proof of Proposition \ref{p:syf}]
First, we note that condition (MP) from Definition \ref{d:MP+SF} follows exactly as in \cite{FP-23}, since it is not affected directly by the nonlinear character of the velocity interactions. We omit this part of the proof and proceed with condition (SF).

\medskip
\noindent
    Pursuing condition (SF), our goal is to construct a set $A\subset [0,T]$ of full measure such that for each $t_0 \in A$ and $0\le \phi \in C^\infty_c (\R^d)$ the family of Radon measures $\{ \rho_{t_0,t}[\phi]\}_{t \in A}$ is narrowly continuous, i.e. for $t_n \rightarrow t_0 \in A$, $\{t_n\} \subset A$ we have
    \begin{equation}\label{e:goal}
        \int_{\R^d} \xi(x) \dd \rho_{t_0, t_n} [ \phi ](x) \rightarrow \int_{\R^d} \xi(x) \dd \rho_{t_0, t_n}[\phi](x) \quad \forall \xi \in C_b (\R^d).
    \end{equation}
    Before proceeding, let us also introduce the following cut-off function of $t$:
    \begin{equation}\label{cutof}
        \psi_{t_1, t_2, \varepsilon}(t) := 
            \begin{cases}
                0, \quad t\le t_1-\varepsilon \text { or } t\ge t_2 + \varepsilon,\\
                1, \quad t_1+\varepsilon \le t \le t_2 - \varepsilon,\\
                \text{linear on } |t-t_i|\le \varepsilon, \; i=1,2.
            \end{cases}
    \end{equation}
    with $0<t_1<t_2<T$ fixed and $\varepsilon<\min \{t_1, T-t_2 \}$.
    The strategy is as follows: we test the weak formulation of kinetic equation \eqref{e:weakkin} with $\psi_{t_1, t_2, \varepsilon}$ multiplied by functions depending on $x$ and $v$ and use Lemma \ref{l:mu_dissipate} to bound the right-hand side from above. Using the Lebesgue-Besicovitch differentiation theorem, we pass with $\varepsilon\rightarrow 0$ to extract full-measure subsets of the continuity set $A_E$ defined in Lemma \ref{l:mu_dissipate}. In fact, we do so countably many times to cover a dense subset of the required test functions.

    The main problem stems from the fact that in order to obtain (SF) we need to able to take $t_0$ as either $t_1$ or $t_2$ in $\psi_{t_1, t_2, \varepsilon}$. Since we extract the domain for $t_1$ and $t_2$ in a way depending on $t_0$, such ability is far from obvious. This issue is solved by the fact that the function 
    $$
        (t_0, t) \mapsto T^{t_0, t}_\# \mu_t
    $$
    is narrowly continuous with respect to $t_0$ for each fixed $t$.
    \smallskip
    
    $\diamond$ {\sc Step 1.} {\it The case of fixed smooth $\phi$ and $\xi$ and a fixed $t_0$.}
    
    \smallskip
    \noindent We start with the following claim.
    
    {\it \underline{Claim A.} Let $D_E$ be a countable, dense subset of $A_E$, where $A_E$ is as in Definition \ref{d:app_sol} and Lemma \ref{l:mu_dissipate}. Fix $\phi, \xi \in C^\infty_c(\R^d)$ and any $t_0 \in D_E$. Then for any $t_1 \in A_E$ and any $\eta>0$ there exists $\delta = \delta(t_1, \phi, \xi) >0$ such that for all $t_2 \in A_E$, $|t_1 - t_2|<\delta$ we have
    $$
    \left\lvert \int_0^T \int_{\R^{2d}} \partial_t \psi_{t_1, t_2, \varepsilon}(t) \xi(x) \phi(v) \dd T^{t_0, t}_\# \mu_t \dd t \right\rvert \le \eta,
    $$
    where $T^{t_0, t}$ is as in the condition (SF) and $\psi_{t_1, t_2, \varepsilon}$ is defined in \eqref{cutof}.
    }

    To prove the claim, fix $\xi, \psi$ and $t_0$ as in its assumptions and set
    \begin{equation}
        L_{\phi, \xi} := \max \left\{ 
        \lVert\phi \rvert_{\infty}, [\phi]_{\text{Lip}}, \lVert \nabla_v \phi \rVert_{\infty}, [\nabla_v \phi]_{\text{Lip}}, 
        \lVert \xi \rVert_\infty, [\xi]_{\text{Lip}}, \lVert\nabla \xi \rVert_\infty, [\nabla \xi]_{\text{Lip}}.
        \right\}
    \end{equation}
    Fix arbitrary $t_1 \in A_E$ and $\eta >0$. Then, by the continuity established in Lemma \ref{l:mu_dissipate}, there exists a $\delta>0$ (which we fix from now on) such that 
    \begin{equation}\label{e:eq3}
        \forall t \in A_E \quad |t-t_1|<\delta \Rightarrow \left| E[\mu_t] - E[\mu_{t_1}] \right| < \frac{\eta}{4}.
    \end{equation}
    With $t_1$ already fixed, choose any $t_2 \in A_E$ such that $|t_1-t_2|<\delta$. For the sake of readability, denote $\psi_{\varepsilon}(t) = \psi_{t_1, t_2, \varepsilon}(t)$ and consider the test function
    \begin{equation*}
        \phi_{\varepsilon}(t,x,v) := \psi_\varepsilon (t) \xi \left( \pi_x \left( T^{t_0, t}(x,v) \right) \right) \phi(v),
    \end{equation*}
    where $\pi_x$ is the projection onto $x-$coordinate. We test weak formulation of kinetic equation \eqref{e:weakkin} for the approximative solutions $\mu^N$, obtaining
    \begin{equation}\label{e:eq2}
        \int_0^T \int_{\R^{2d}} \partial_t \phi_\varepsilon + v \cdot \nabla_x \phi_\varepsilon \dd \mu^N_t \dd t = 
        \int_{t_1 - \epsilon}^{t_2 + \epsilon} \int_{\R^{4d}\setminus \Delta} \underbrace{\frac{\left( \nabla_v \psi_\varepsilon - \nabla_v \psi_\varepsilon'\right)\cdot(v-v')|v-v'|^{p-2}}{|x-x'|^\alpha}}_{\mathcal{G}:=}\dd [\mu_t(x,v) \otimes \mu_t(x',v')]\dd t,
    \end{equation}
    where we assumed without the loss of generality that $t_1<t_2$. We estimate the integrand $\mathcal{G}$ as follows:
    \begin{align*}
        |\mathcal{G}| &=
        \psi_\epsilon(t) \left\lvert
        \frac{\left( \nabla_v \left( \xi (x-(t-t_0)v) \phi(v)\right) - \nabla_v \left( \xi (x'-(t-t_0)v') \phi(v')\right) \right)\cdot(v-v')|v-v'|^{p-2}}{|x-x'|^\alpha}
        \right\rvert \\
        &\le \psi_\epsilon(t) \left\lvert 
        \frac{\left( \nabla\xi(x-(t-t_0)v)(t-t_0)\phi(v) - \nabla \xi (x' - (t-t_0)v')(t-t_0)\phi(v')\right)\cdot (v-v')|v-v'|^{p-2}}{|x-x'|^\alpha} 
        \right\rvert \\
        &\quad + \psi_\epsilon(t) \left\lvert 
        \frac{\left( \xi (x-(t-t_0)v)\nabla_v \phi(v) - \xi (x'-(t-t_0)v')\nabla_v \phi(v')\right)\cdot (v-v')|v-v'|^{p-2}}{|x-x'|^\alpha}
        \right\rvert
        =: I + II.
    \end{align*} 
    Now up to constants depending on $L_{\phi, \xi}$ and $T$, but independent of $t_0$, we have following further estimates:
    \begin{align*}
        I &\lesssim 
        \psi_\epsilon(t) \left\lvert
            \nabla \xi (x-(t-t_0)v) \frac{(\phi(v) - \phi(v'))\cdot (v-v')|v-v'|^{p-2}}{|x-x'|^\alpha}
        \right\rvert \\
        &\quad + \psi_\epsilon(t) \left\lvert
            \phi(v') \frac{\left( \nabla\xi(x-(t-t_0)v) - \nabla\xi(x'-(t-t_0)v')\right)\cdot (v-v')|v-v'|^{p-2}}{|x-x'|^\alpha}
        \right\rvert \\
        &\lesssim
        \psi_\epsilon(t) \frac{|v-v'|^p}{|x-x'|^\alpha} + \psi_\epsilon \frac{|x-x'||v-v'|^{p-1}}{|x-x'|^\alpha}
    \end{align*}
    and similarly 
    \begin{align*}
        II &\lesssim 
        \psi_\epsilon \left\lvert 
            \psi(x-(t-t_0)v) \frac{\left( \nabla_v \phi(v) - \nabla_v \phi(v')\right)\cdot (v-v')|v-v'|^{p-2}}{|x-x'|^\alpha}
        \right\rvert \\
        &\quad + \psi_\epsilon(t) \left\lvert
            \nabla_v \phi(v') \frac{\left( \psi(x-(t-t_0) v) - \psi(x'-(t-t_0)v')\right)\cdot (v-v')|v-v'|^{p-2}}{|x-x'|^\alpha}
        \right\rvert \\
        &\lesssim \psi_\epsilon \frac{|v-v'|^p}{|x-x'|^\alpha} + \psi_\epsilon(t) \frac{|x-x'||v-v'|^{p-1}}{|x-x'|^\alpha}.
    \end{align*}
    Arguing as in the proof of \ref{lem:mutimesmu}, we further bound the second term in both inequalities by Young's inequality with exponents $1/p+1/p'=1$, that is 
    \begin{align*}
        |x-x'|^{1-\alpha}|v-v'|^{p-1} &= |x-x'|^{1-\frac{\alpha}{p}} \frac{|v-v'|^{p-1}}{|x-x'|^{\alpha/p'}} \le 
        \frac{1}{p}|x-x'|^{p - \alpha} + \left( 1 + \frac{1}{p'}\right) \frac{|v-v'|^p}{|x-x'|^\alpha}
    \end{align*}
    obtaining 
    \begin{equation}\label{e:eq1}
        \left \lvert \mathcal{G} \right\rvert \le C_{\phi, \xi} \psi_{\varepsilon}(t) \frac{|v-v'|^p}{|x-x'|^\alpha} + C_{\phi, \xi} \psi_{\varepsilon}(t) |x-x'|^{p-\alpha},
    \end{equation}
    where $C_{\phi, \xi}$ is a constant depending on $T$ and $L_{\phi, \xi}$ only. To further estimate $\psi_{\varepsilon}$ observe that since $|t_2 - t_1|< \delta$, for sufficiently small $\epsilon > 0$ (i.e. such that $|t_2 - t_1| < \delta - \epsilon$) we have spt$(\psi_\varepsilon) \subset [t_1 - \epsilon, t_2 + \epsilon] \subset (t_1 - \delta, t_1 + \delta)$. Since $A_E$ is of full measure, there exist $t_1'$, $t_2' \in A_E$ such that $t_1' \le t_1 - \varepsilon < t_2 - \varepsilon < t_2'$ with $|t_i' - t_1|<\delta$ still. Thus for sufficiently small $\epsilon>0$ we have $\psi_{\varepsilon}(t) \le \chi_{[t_1', t_2']}(t)$ which allows us to further estimate \eqref{e:eq1}:
    \begin{equation*}
        |\mathcal{G}| \le C_{\phi, \xi} \chi_{[t_1', t_2']}(t) \frac{|v-v'|^p}{|x-x'|^\alpha} + C_{\phi, \xi} \chi_{[t_1', t_2']}(t) |x-x'|^{p-\alpha}.
    \end{equation*}
    Going back to $\eqref{e:eq2}$ and recalling that $\mu_t$ is compactly supported we have 
    \begin{align*}
        \left \lvert \int_0^T \int_{\R^{2d}} \partial_t \phi_\epsilon + v \cdot \nabla_x \phi_\varepsilon \dd \mu_t^N \dd t \right \rvert 
        &\le C_{\phi, \xi} \int_{t_1'}^{t_2'} D_p [\mu_t^N] \dd t + C_{\phi, \xi} \int_{t_1'}^{t_2'}|x-x'|^{p-\alpha}\dd \left( \mu_t^N \otimes \mu_t^N \right) \dd t \\
        &\le C_{\phi, \xi} \left\lvert E[\mu_{t_1'}^N] - E[\mu_{t_2'}^N] \right\rvert + C_{\phi, \xi} \left\lvert 2(T+1)M\right\rvert^{p-\alpha} |t_2' - t_1'|,
    \end{align*}
    where in the last inequality we used proposition \ref{p:eneq} and the fact that $\mu_t^N$ is uniformly compactly supported. By definition \ref{d:app_sol} we know that $E[\mu_t^N]\xrightarrow{N \rightarrow \infty} E[\mu_t]$ on $A_E \ni {t_1', t_2'}$. Therefore we can pass with $N\rightarrow \infty$ in the inequality above, obtaining
    \begin{align*}
        \frac{1}{C_{\phi, \xi}} \left\lvert \int_0^T \int_{\R^{2d}} \partial_t \phi_\epsilon + v \cdot \nabla_x \phi_\epsilon \dd \mu_t \dd t \right\rvert
        &\le \left\lvert E[\mu_{t_1'}]-E[\mu_{t_2'}] \right\rvert + |2(T+1)M|^{p-\alpha}|t_2' - t_1'| \\
        &\le \left\lvert E[\mu_{t_1'}]-E[\mu_{t_1}] \right\rvert + \left\lvert E[\mu_{t_1}]-E[\mu_{t_2'}] \right\rvert \\
        &\quad +|2(T+1)M|^{p-\alpha}|t_1' - t_1'| + |2(T+1)M|^{p-\alpha}|t_2' - t_1|.
    \end{align*}
    Recalling \eqref{e:eq3} and that $|t_i' - t_1| < \delta$ for $i=1,2$, after possibly multiplying $\eta$ by constant depending on $L_{\phi,\xi}$ and $T$, we arrive at 
    \begin{equation*}
        \left\lvert \int_0^T \int_{\R^{2d}} \partial_t \phi_\varepsilon + v \cdot \nabla_x \phi_\epsilon \dd \mu_t \dd t \right\rvert \le C_{\phi,\xi} \frac{\eta}{2} + 2 C_{\phi, \xi} |2(T+1)M|^{p-\alpha} \delta.
    \end{equation*}
    Observe that for $\alpha \le p $ we can always make $\delta$ smaller if needed to obtain 
    \begin{equation}\label{e:eq5}
        \left\lvert \int_0^T \int_{\R^{2d}} \partial_t \phi_\epsilon + v\cdot \nabla_x \phi_\epsilon \dd \mu_t \dd t \right\rvert \le C_{\phi, \xi} \eta.
    \end{equation}
    To conclude the first step, let us expand the left-hand side of the above inequality. It reads
    \begin{equation}\label{e:eq4}
        \begin{aligned}
            &\int_0^T \int_{\R^{2d}} \partial_t \phi_\epsilon + v\cdot \nabla_x \phi_\epsilon \dd \mu_t \dd t \\
            &\quad = \int_0^T \int_{\R^{2d}} \partial_t \phi_\epsilon (t) \xi (x-(t-t_0)v)\phi(v) \dd \mu_t \dd t \\ 
            &\qquad + \int_0^T \phi_\epsilon (t) \int_{\R^{2d}} \underbrace{\partial_t \left( \xi(x-(t-t_0)v)\right)\phi(v) + v\cdot \nabla \xi (x-(t-t_0)v)\phi(v)}_{=0} \dd\mu_t \dd t \\ 
            &\quad =\int_0^T \int_{\R^{2d}} \partial_t \psi_\varepsilon(t) \xi(x) \phi(v) \dd T^{t_0,t}_\# \mu_t \dd t.
        \end{aligned}
    \end{equation}
    Combining \eqref{e:eq4} with \eqref{e:eq5}, we conclude the proof of the claim and the first step.

    \smallskip
    
    $\diamond$ {\sc Step 2.} {\it Passing with $\epsilon \rightarrow 0$}.

    \smallskip

    By Definition \ref{d:app_sol}, both $\{\mu_t\}_{t\in[0,T]}$ and $\{T^{t_0,t}_\# \mu_t \}_{t \in [0,T]}$ are uniformly compactly supported with the support contained in some large compact set $K\times K \subset \R^{2d}$. By Stone-Weierstrass theorem, there exists a countable set $\mathcal{D}\subset C^\infty_c (\R^d)$ dense in $C_b(K)$. With that fact at hand, we aim to prove the following claim.

    {\it \underline{Claim B.} There exists a full measure set $A\subset [0,T]$ such that for all $t_0\in D_E$ and all $\phi, \xi \in \mathcal{D}$ the following statements are true.\\
    For all $t_1 \in A$ and all $\eta > 0$ there exists $\delta>0$ (depending on $t_1, \phi$ and $\xi$) such that for all $t_2 \in A$ satisfying $|t_2 - t_1| < \delta$ we have 
    $$
    \left\lvert \int_{\R^{2d}} \xi(x) \phi(v) \dd T^{t_0,t_2}_\# \mu_{t_2} - \int_{\R^{2d}} \xi(x)\phi(v) \dd T^{t_0, t_1}_\# \mu_{t_1} \right\rvert \le \eta.
    $$
    Equivalently, for all $t_1\in A$ and all $\{t_n\}_{n=1}^\infty \subset A$ such that $t_n \rightarrow t_1$
    $$
    \int_{\R^{2d}} \xi(x) \phi(v) \dd T^{t_0, t_n}_\# \mu_{t_n} \xrightarrow{n \rightarrow \infty} \int_{\R^{2d}} \xi(x) \phi(v) \dd T^{t_0, t_1}_\# \mu_{t_1}.
    $$
    }
    Take any $t_0 \in D_E$ and any $\phi_1, \xi_1 \in \mathcal{D}$. By Claim A, we know that for all $t_1 \in A_E$ and $\eta>0$, there exists $\delta>0$ (importantantly, independent of $\epsilon$) such that for $t_2 \in A_E$, $|t_1-t_2|<\delta$ we have
    \begin{equation}\label{e:eq6}
    \left\lvert \int_0^T \int_{\R^{2d}} \partial_t \psi_{t_1, t_2, \varepsilon}(t) \xi_1(x) \phi_1(v) \dd T^{t_0, t}_\# \mu_t \dd t \right\rvert \le \eta.
    \end{equation}
    We aim to pass with $\epsilon \rightarrow 0$ in the inequality above. We start by differentiating the integrand with respect to $t$; since only $\psi$ is $t-$dependent, this amounts to observing that by \eqref{cutof} we have 
    \begin{equation}\label{d/dt_cutof}
         \psi_{t_1, t_2, \varepsilon}(t) := 
            \begin{cases}
                (2\epsilon)^{-1} \quad \text{for } t \in (t_1 - \varepsilon, t_1 + \varepsilon), \\
                -(2\epsilon)^{-1} \quad \text{for } t \in (t_2 - \varepsilon, t_2 + \varepsilon),\\
                0 \quad \text{otherwise.}
            \end{cases}
    \end{equation}
    Thus \eqref{e:eq6} translates into 
    \begin{equation*}
        \left\lvert \frac{1}{2\epsilon} \int_{t_1 - \epsilon}^{t_1+\epsilon} \int_{\R^{2d}} \xi_1(x) \phi_1(v) \dd T^{t_0,t}_\#\mu_t  \dd t 
        - \frac{1}{2\epsilon} \int_{t_2 - \varepsilon}^{t_2 + \varepsilon} \int_{\R^{2d}} \xi_1(x) \phi_1(v) \dd T^{t_0,t}_\# \mu_t \dd t\right\rvert \le \eta.
    \end{equation*}
    To pass with $\epsilon\rightarrow 0$, observe that the function $t \mapsto \int_{\R^{2d}} \xi_1(x) \phi_1(v) \dd T^{t_0,t}_\# \mu_t$ is integrable since $\phi_1, \xi_1$ are bounded. Therefore we can apply the Lebesgue-Besicovitch differentiation theorem to ensure the existence of a full-measure subset $A_{1, t_0} \subset A_E$, depending on $\xi_1, \phi_1$ and $t_0$, such that for all $t_1,t_2 \in A_{1,t_0}$
    \begin{align*}
            \lim\limits_{\epsilon \rightarrow 0} \left\lvert \int_0^T \int_{\R^{2d}} \partial_t \psi_{t_1, t_2, \varepsilon}(t) \xi_1(x) \phi_1(v) \dd T^{t_0, t}_\# \mu_t \dd t \right\rvert
    &= \left\lvert \int_{\R^{2d}} \xi_1(x) \phi_1(v) \dd T^{t_0, t_1}_\# \mu_{t_1} - \int_{\R^{2d}} \xi_1(x) \phi_1(v) \dd T^{t_0, t_2}_\# \mu_{t_2} \right\rvert \\
    &\le \eta
    \end{align*}
    (where we still assume $|t_1 - t_2|<\delta$ as in Claim A). This procedure can be repeated indefinitely for all $i=1,2,\ldots$ and $\phi_i, \xi_i \in \mathcal{D}$ as well as for all $t_0 \in D_E$ (which is a countable set), leading to the existence of a full measure subset
    \begin{equation}\label{A_def}
        A := \bigcap_{i \in \mathbb{N}, t_0 \in D_E} A_{i,t_0} \subset A_E
    \end{equation}
    satisfying claim B.

    \smallskip
    
    $\diamond$ {\sc Step 3.} {\it Taking $t_0 = t_1$}.

    \smallskip
    While Claim B is already close to the condition (SF), we still need to make sure that we can take $t_0 = t_1$. It is not immediately clear, since claim B holds only for $t \in A$ and the set $A$ may be disjoint with $D_E \ni t_0$.  Nevertheless, observe that for any $t_0 \in A$, $t_n \rightarrow t_0$, $\{t_n\}_{n=1}^\infty \subset D_E$ and any $\phi, \xi \in \mathcal{D}$ we have
    \begin{equation}\label{e:eq7}
    \left\lvert \int_{\R^{2d}} \xi \phi \dd T^{t_n,t_1}_\# \mu_{t_1} - \int_{\R^{2d}} \xi \phi \dd T^{t_0,t_1}_\# \mu_{t_1} \right\rvert \le M [\xi]_{\text{Lip}}\lVert \phi \rVert_{\infty}|t_n - t_0|,
    \end{equation}
    meaning that for such $t_1, \phi, \xi$ the function
    $$
    t_0 \mapsto \int_{\R^{2d}} \xi \phi \dd T^{t_0, t_1}_\# \mu_{t_1}
    $$
    is continuous (in fact, we have a uniform in $t_1$ Lipschitz continuity). With that in mind, we aim to prove the next claim.\\
    {\it \underline{Claim C.} For all $t_0 \in A$, $\phi, \xi \in \mathcal{D}$ and any $\{t_n \}_{n=1}^\infty$ such that $t_n \rightarrow t_0$ we have 
    $$
    \int_{\R^{2d}} \xi \phi \dd T^{t_0, t_n}_\# \mu_{t_n} \rightarrow \int_{\R^{2d}} \xi \phi \dd T^{t_0,t_0}_\# \mu_{t_0}.
    $$
    }
    Let $t_0, \phi, \xi$ and $\{t_n\}_{n=1}^\infty$ satisfy assumptions of the Claim. Then for all $\eta>0$ there exists $t_0' \in D_E$ such that $|t_0 - t_0'|< \eta$. By Claim B, we have 
    $$
    \int_{\R^{2d}} \xi \phi \dd T^{t_0', t_n}_\# \mu_{t_n} \rightarrow \int_{\R^{2d}} \xi \phi \dd T^{t_0, t_0}_\# \mu_{t_0}.
    $$
    Exploiting \eqref{e:eq7}, we arrive at 
    \begin{align*}
        &\left|\int_{\R^{2d}}\xi\phi  \dd T^{t_0,t_n}_\#\mu_{t_n} - \int_{\R^{2d}}\xi\phi  \dd T^{t_0,t_0}_\#\mu_{t_0}\right| \\
        &\quad \leq \left|\int_{\R^{2d}}\xi\phi  \dd T^{t_0,t_n}_\#\mu_{t_n} - \int_{\R^{2d}}\xi\phi  \dd T^{t_0',t_n}_\#\mu_{t_n}\right|
    + \left|\int_{\R^{2d}}\xi\phi  \dd T^{t_0',t_n}_\#\mu_{t_n} - \int_{\R^{2d}}\xi\phi  \dd T^{t_0',t_0}_\#\mu_{t_0}\right|\\
    &\qquad + \left|\int_{\R^{2d}}\xi\phi  \dd T^{t_0',t_0}_\#\mu_{t_0} - \int_{\R^{2d}}\xi\phi  \dd T^{t_0,t_0}_\#\mu_{t_0}\right|\\
    & \quad \leq \left|\int_{\R^{2d}}\xi\phi  \dd T^{t_0',t_n}_\#\mu_{t_n} - \int_{\R^{2d}}\xi\phi  \dd T^{t_0',t_0}_\#\mu_{t_0}\right|+ 2M[\xi]_{Lip}|\phi|_\infty\eta \xrightarrow{n\to\infty} 2M[\xi]_{Lip}|\phi|_\infty\eta. 
    \end{align*}
    Since $\eta>0$ was arbitrary, Claim C is proved.
    \smallskip
    
    $\diamond$ {\sc Step 4.} {\it Relaxing the restrictions on $\phi$ and $\xi$}.

    \smallskip
    Our last effort is to generalise all the previous steps to arbitrary $\phi \in C^\infty_c (\R^d)$, $\xi \in C_b(\R^d)$. This follows from the usual density argument; let us fix any $\phi \in C^\infty_c (\R^d)$, $\xi \in C_b(\R^d)$ and let $\{t_n\}_{n=1}^\infty$ be such that $t_n \rightarrow t_0$. Aiming at \eqref{e:goal}, our goal is to establish that 
    \begin{equation}\label{e:goal_bis}
        \int_{\R^d} \xi \dd \rho_{t_0, t_n}[\phi] \rightarrow \int_{\R^d} \xi \dd \rho_{t_0,t_0}[\phi].
    \end{equation}
    Keeping the definition of $\mathcal{D}$ and $K$ from the Step 2, let us fix $\eta>0$ and functions $\phi_{\eta}, \xi_{\eta} \in \mathcal{D}$ such that 
    $$
    \sup\limits_{x \in K} \left\lvert \xi(x) - \xi_\eta(x) \right\rvert + \sup\limits_{v \in K} \left\lvert \phi(v) - \phi_\eta(v)\right\rvert < \eta.
    $$
    Then by Claim C, the following convergence holds:
    \begin{align*}
        &\left|\int_{\R^d}\xi\dd \rho_{t_0,t_n}[\phi] - \int_{\R^d}\xi\dd \rho_{t_0,t_0}[\phi]\right| \\
        &\quad =  \left|\int_{\R^{2d}}\xi\phi\dd T^{t_0,t_n}_\#\mu_{t_n}  - \int_{\R^{2d}}\xi\phi\dd T^{t_0,t_0}_\#\mu_{t_0}\right|\\
        & \quad \leq \left|\int_{\R^{2d}}\xi_\eta\phi_\eta\dd T^{t_0,t_n}_\#\mu_{t_n} - \int_{\R^{2d}}\xi_\eta\phi_\eta\dd T^{t_0,t_0}_\#\mu_{t_0}\right| + \left|\int_{\R^{2d}}(\xi_\eta\phi_\eta - \xi\phi)\dd T^{t_0,t_n}_\#\mu_{t_n}\right| \\
        &\qquad + \left|\int_{\R^{2d}}(\xi_\eta\phi_\eta- \xi\phi)\dd T^{t_0,t_0}_\#\mu_{t_0}\right|\\
        & \quad \lesssim \left|\int_{\R^{2d}}\xi_\eta\phi_\eta\dd T^{t_0,t_n}_\#\mu_{t_n} - \int_{\R^{2d}}\xi_\eta\phi_\eta\dd T^{t_0,t_0}_\#\mu_{t_0}\right| + \eta(|\xi|_\infty+|\phi|_\infty)\xrightarrow{n\to\infty} \eta(|\xi|_\infty+|\phi|_\infty).
    \end{align*}
    Due to the arbitrarily of $\eta>0$, convergence \eqref{e:goal_bis} holds for any $\phi \in C^\infty_c (\R^d)$ and $\xi \in C_b(\R^d)$. The proof of proposition \ref{p:syf} is finished.
    
\end{proof}

{\footnotesize
\bibliographystyle{abbrv}
\bibliography{main}
}

\end{document}